\documentclass[11pt]{extarticle}

\usepackage[english]{babel}
\usepackage{amsmath}
\usepackage[colorlinks=true, allcolors=blue]{hyperref}
\usepackage{amssymb}
\usepackage{enumitem}
\setlist{leftmargin=5.5mm}
\usepackage{amsthm}
\usepackage[dvipsnames]{xcolor}%Colors
\usepackage{tikz}%Diagrams
\usepackage{caption}
\usepackage{subcaption}
\usepackage[margin=1in]{geometry}
\newtheorem{theorem}{Theorem}[section]
\theoremstyle{definition}
\newtheorem{definition}[theorem]{Definition}

\newtheorem{lemma}[theorem]{Lemma}
\newtheorem{proposition}[theorem]{Proposition}
\newtheorem{corollary}[theorem]{Corollary}
\theoremstyle{definition}
\newtheorem{example}[theorem]{Example}
\newtheorem{algorithm}[theorem]{Algorithm}
\newtheorem{conjecture}[theorem]{Conjecture}

\newcommand{\Seq}{\textsc{Seq}}
\newcommand{\USeq}{\textsc{USeq}}

\newcommand{\am}{\mathrm{m}} 

\renewcommand{\*}[1]{\mathbf{#1}}

\title{$\mathbf{\textit{k}}$-NIM Trees: Characterization and Enumeration\footnote{This work was carried out during the summer of 2021 and was supported by the NSF Grant DMS \#0751964.}}

\author{Charles R. Johnson\footnote{Department of Mathematics, College of William and Mary, 
Williamsburg, VA. crjohn@wm.edu}, George Tsoukalas\footnote{Department of Mathematics, Rutgers University, New Brunswick, NJ; \tt george.tsoukalas@rutgers.edu}, 
Greyson C. Wesley\footnote{Department of Mathematics, University of Notre Dame, Notre Dame, IN;
\tt gwesley@nd.edu},
Zachary Zhao\footnote{Department of Mathematics, Washington University in St. Louis, St. Louis, MO; \tt zlzhao@wustl.edu}}

\date{} %leave blank

\begin{document}

% \begin{frontmatter}
\maketitle

\begin{abstract}
    Among those real symmetric matrices whose graph is a given tree $T$, the maximum multiplicity $M(T)$ that can be attained by an eigenvalue is known to be the path cover number of $T$. We say that a tree is $k$-NIM if, whenever an eigenvalue attains a multiplicity of $k-1$ less than the maximum multiplicity, all other multiplicities are $1$. $1$-NIM trees are known as NIM trees, and a characterization for NIM trees is already known. Here we provide a graph-theoretic characterization for $k$-NIM trees for each $k\geq 1$, as well as count them. It follows from the characterization that $k$-NIM trees exist on $n$ vertices only when $k=1,2,3$. In case $k=3$, the only $3$-NIM trees are simple stars.
\end{abstract}



% \noindent\authorroles{For determining author roles, please use following taxonomy: \url{https://casrai.org/credit/}. Please list the roles for each author.}

% \newpage

\section{Introduction}
Let $T$ be a tree on $n$ vertices. We consider the set of real symmetric matrices associated with $T$, denoted $\mathcal{S}(T)$, in which $A=(a_{i j}) \in \mathcal{S}(T)$ if and only if $a_{i j} \neq 0$ for each edge $\{i, j\}$ of $T$, $i \neq j$. We are interested in the set of unordered eigenvalue multiplicity lists of matrices in $\mathcal{S}(T)$, called the \textit{catalog} of $T$, denoted $\mathcal{L}(T)$. An unordered multiplicity list lists, in descending order, the multiplicities of the eigenvalues of a matrix. The order of the multiplicities does not depend on the numerical values of the underlying eigenvalues.

There has been considerable study of the possible unordered multiplicity lists of real symmetric matrices in $\mathcal{S}(T)$. Given a matrix $A \in \mathcal{S}(T)$ and a real number $\lambda$, denote by $\am_{A}(\lambda)$ the multiplicity of $\lambda$ in $A$ (since $A \in \mathcal{S}(T)$, algebraic and geometric multiplicity are the same). Given a tree $T$ on $n$ vertices and a set $\alpha \subseteq\{1, \ldots, n\}$, denote by $A[\alpha]$ the principal submatrix of $A$ resulting from including only the rows and columns of $A$ indexed by $\alpha$. Similarly, denote by $A(\alpha)$ the principal submatrix resulting from excluding the rows and columns of $A$ indexed by $\alpha$. If $v$ is a vertex of $T$, then $A[v]$ and $A(v)$ are defined analogously. We call a vertex of degree $3$ or higher a \textit{high degree vertex}, or HDV.

Given a matrix $A$ whose graph is $T$, a vertex is called \textit{Parter} for an eigenvalue $\lambda$ if $\am_{A(v)}(\lambda)=\am_{A}(\lambda)+1$. Similar, a \textit{Parter set} is a set of vertices $Q$ such that $\am_{A(Q)}(\lambda)=\am_{A}(\lambda)+|Q|$. The Parter-Wiener, etc. Theorem, which reveals that all multiple eigenvalues (and some eigenvalues of multiplicity 1) have at least one Parter vertex. Note that the removal of a vertex can never change the multiplicity by more than $1$ because of the Cauchy interlacing theorem.

\begin{theorem}\label{parterwiener}
% [Parter-Wiener, etc.]

\cite{johnson_duarte_saiago_2003}
Let $T$ be a tree, $A\in\mathcal{S}(T)$ and suppose there is a vertex $v$ of $T$ and a real number $\lambda$ such that $\lambda\in\sigma(A)\cap\sigma(A(v))$. Then

\begin{enumerate}
\item there is a vertex $u$ of $T$ such that $\am_A(\lambda)=\am_A(\lambda)+1$, i.e., $u$ is Parter for $\lambda$ (with respect to $A$ and $T$);

\item if $\am_A(\lambda)\geq 2$, then the prevailing hypothesis is automatically satisfied and $u$ may be chosen so that $\deg_T(u)\geq 3$ and so that there are at least three components $T_1$, $T_2$, and $T_3$ of $T-u$ such that $\am_{A[T_i]}(\lambda)\geq 1$, $i=1,2,3$; and 

\item $\am_A(\lambda)=1$, then $u$ may be chosen so that $\deg_T(u)\geq 2$ so that there are two components $T_1$ and $T_2$ of $T-u$ such that $\am_{A[T_i]}(\lambda)=1$, $i=1,2$.
\end{enumerate}
\end{theorem}

We say a vertex $i$ is a \textit{strong Parter} vertex of $T$ for $\lambda$ relative to $A$ (a strong Parter vertex, for short) if $v$ has degree at least $3$ in $T$ such that $\am_{A(i)}(\lambda)=\am(\lambda)+1$ and $\lambda$ occurs as an eigenvalue in direct summands of $A(i)$ that correspond to at least three branches of $T$ at $i$.

The \emph{maximum multiplicity} $M(T)$ is the largest entry of any list in the catalog $\mathcal{L}(T)$. It is known that $M(T)$ is equal to the \emph{path cover number} $P(T)$, the minimum number of vertex disjoint induced paths that contain all vertices of $T$ (Theorem 3.4.3 of \cite{johnson_saiago_2018}). Together with Theorem \ref{parterwiener}, this information implies the following two results:

\begin{theorem}[Theorem 4.2.3 from \cite{johnson_saiago_2018}] \label{monograph423}
Let $T$ be a tree, $A\in\mathcal{S}(T)$ and $\lambda$ be an eigenvalue of $A$. Then there is a vertex $v$ of $T$ such that $\lambda \in \sigma(A) \cap \sigma(A(v))$ if and only if there is a Parter set $Q$ of cardinality $k \geq 1$ such that $\lambda$ is an eigenvalue of $\am_A(\lambda) + k$ direct summands of $A(Q)$ (with multiplicity 1 in each). Moreover, if vertex $v$ above is Parter for $\lambda$, then such a Parter set $Q$ may be constructed so that $v \in Q$.
\end{theorem}

\begin{theorem}[Theorem 4.2.9 from \cite{johnson_saiago_2018}]\label{monograph429} Let $T$ be a tree, $A\in\mathcal{S}(T)$ and $\lambda$ be a multiple eigenvalue of $A$. There is a fragmenting Parter set for $\lambda$ relative to $A$ whose elements are all HDVs. Moreover, if $Q$ is a fragmenting Parter set of minimum cardinality for a multiple eigenvalue, then each vertex of $Q$ is an HDV.
\end{theorem}

We call a Parter set $Q$, as guaranteed by \ref{monograph423}, a \emph{fragmenting Parter set} of vertices of $T$ relative to $A$ (a fragmenting Parter set, for short).

Theorem \ref{parterwiener} is also key to giving eigenvalue assignments for multiplicity lists.

\begin{definition}[Assignment] Let $T$ be a tree on $n$ vertices and let $(p_1,p_2,\dots,p_k,(1)^{n-\sum_{i=1}^n p_i})$ be a non-increasing list of positive integers, with $\sum_{i=1}^{k} p_{i} \leq n$. Then, an \emph{assignment} $\mathcal{A}$ is a collection $\mathcal{A}=\left\{\mathcal{A}_{1}, \ldots, \mathcal{A}_{k}\right\}$ of $k$ collections $\mathcal{A}_{i}$ of subtrees of $T$, corresponding to eigenvalues with multiplicities $\am_{i}(A)$, with the following properties:

\begin{itemize}
\item[(1)] (Specification of Parter vertices) For each $i$, there exists a set $V_{i}$ of vertices of $T$ such that
\begin{itemize} 
\item[(1a)]each subtree in $\mathcal{A}_{i}$ is a connected component of $T-V_{i}$, 
\item[(1b)] $\left|\mathcal{A}_{i}\right|=p_{i}+\left|V_{i}\right|$, and 
\item[(1c)] for each vertex $v \in V_{i}$, there exists a vertex $x$ adjacent to $v$ such that $x$ is in one of the subtrees in $\mathcal{A}_{i}$.
\end{itemize}
\item[(2)] (No overloading) We require that no subtree $S$ of $T$ is assigned more than $|S|$ eigenvalues; define $c_{i}(S)=\left|\mathcal{A}_{i} \cap \mathcal{Z}(S)\right|-\left|V_{i} \cap S\right|$, the difference between the number of subtrees contained in $S$ and the number of Parter vertices in $S$ for the $i$th multiplicity. Then we require that $\sum_{i=1}^{k} \max \left(0, c_{i}(S)\right) \leq|S|$ for each $S \in \mathcal{Z}(T)$. If this condition is violated at any subtree, then that subtree is said to be \emph{overloaded}.

\underline{Notation}: If $V$ is a set of vertices and $G$ is a graph, then $V\cap G$ denotes the set of vertices in both $V$ and $G$. Additionally, if $T$ is a tree, we let $\mathcal{Z}(T)$ denote the collection of all subtrees of $T$, including $T$, rather than the power set of the vertices in $T$.
\end{itemize} 
We also refer to the $i$th eigenvalue as being ``assigned'' to each subtree in $\mathcal{A}_{i}$.
\end{definition}

(Further developments of the use of assignments, and some variations of an assignment, may be found in \cite{johnson_jordan-squire_sher_2010}.)

In practice, the usage of assignments is simpler than the definition suggests. We record the full definition for completeness. We also make use of some related concepts:

\begin{definition}[Assignment Candidate/Near-Assignment]
An \emph{assignment candidate} is a collection of vertices and components satisfying condition (1), but not necessarily (2). Similarly, a \emph{near-assignment} is a collection of vertices and components satisfying conditions (1a), (1b), and (2), but not necessarily (1c). We also define a near-assignment candidate to be a similar collection satisfying (1a) and (1b) but not necessarily (1c) or (2).
\end{definition}

Given a list $p=(p_{1}, p_{2}, \dots, p_{k}, 1^{n-\sum_{i=1}^{k} p_{i}})$, we call an assignment $\mathcal{A}$ of $p$ for a tree $T$ \textit{realizable} if there exists a matrix $B \in \mathcal{S}(T)$ with multiplicity list $(p_{1}, p_{2}, \dots, p_{k}, 1^{n-\sum_{i=1}^{k} p_{i}})$ and eigenvalues $\lambda_{1}, \lambda_{2}, \dots, \lambda_{k}$ corresponding to the $p_{i}$, such that the following is true for each $i$ between 1 and $k$ :
(1) For each subtree $R$ of $T$ in $\mathcal{A}_{i}, \lambda_{i}$ is a multiplicity 1 eigenvalue associated with $R$. Also, for each connected component $Q$ of $T-V_{i}$ that is not in $\mathcal{A}_{i}$, $\lambda_{i}$ is not an eigenvalue of $B[Q]$.
(2) For each vertex $v$ in $V_{i}, v$ is Parter for $\lambda_{i}$.
(3) All eigenvalues of $B$ other than the $\lambda_{i}$ have multiplicity 1 .
In this case, we also call the multiplicity list $(p_{1}, p_{2}, \dots, p_{k}, 1^{n-\sum_{i=1}^{k} p_{i}})$ realizable.

\begin{theorem}[from \cite{fernandes_2015}]\label{singlehighmultassignment}
Every valid assignment for one (high multiplicity) eigenvalue in $T$ is realizable. 
\end{theorem}
\begin{lemma}[from \cite{johnson_jordan-squire_sher_2010}]\label{M2thm2pt5}

Given a tree $T$ on $n=p_{1}+p_{2}+l$ vertices, a multiplicity list $(p_{1}, p_{2}, 1^{l})$, a near-assignment of this list for $T$, and any distinct real numbers $\alpha$ and $\beta$, there exists an $A \in \mathcal{S}(T)$ which satisfies the following conditions:

If $R$ is a connected component of $T-V_{1}, \alpha$ is an eigenvalue of $R$ if and only if $R \in \mathcal{A}_{1}$. Similarly, if $S$ is a connected component of $T-V_{2}, \beta$ is an eigenvalue of $S$ if and only if $S \in \mathcal{A}_{2}$.
\end{lemma}

\begin{corollary}[from \cite{fernandes_2015}]\label{alltwoeigassignmentsarerealizable}
Any valid assignment for two (high multiplicity) eigenvalues of multiplicities $p_1$ and $p_2$ in $T$ is realizable.
\end{corollary}

\begin{definition}[RPM Set]
A \emph{residual path maximizing set} (RPM set) for a tree $T$ is a collection of $q$ vertices of $T$, whose removal from $T$ leaves a forest of $p$ paths in such a way that $p-q$ is maximized. In general, an RPM set of vertices is not unique, even in the value of $q$. We notate this maximum value of $p-q$ as $\Delta(T)$.
\end{definition}
\begin{lemma} \emph{(Corollary 4.2.5 in \cite{johnson_saiago_2018})}. \label{monograph425} 
Let $T$ be a tree, $A\in\mathcal{S}(T)$ and $\lambda$ be an eigenvalue of $A$ of multiplicity $M(T)\geq 2$. A vertex $v$ of $T$ is Parter for an eigenvalue $\lambda$ in $A$ if and only if there is an RPM set $Q$ for $T$, containing $v$, and such that $\am_{A(Q)}(\lambda)=\am_A(\lambda)+|Q|$.
\end{lemma}

\begin{definition}[\texorpdfstring{$k$}{k}-NIM] A tree $T$ is called a \emph{$k$-NIM tree} if $M(T) \geq k$ and if, whenever $A\in\mathcal{S}(T)$ has $\am_A(\lambda)=M(T)-(k-1)\geq 2$ for some $\lambda$, it implies that $\am_A(\alpha)=1$ for all remaining eigenvalues $\alpha\in\sigma(A)$. 

\end{definition}
% \begin{theorem}[Theorem 5.2.1 in \cite{johnson_saiago_2018}]\label{monograph521}
% Let $T$ be a tree. Then $T$ is 1-NIM if and only if for each HDV $v$ of $T$,
% \begin{itemize}\label{nimcharacterization}
%     \item[\textbf{\textsc{(i)}}] at most two components of $T-v$ have more than one vertex, and
%     \item[\textbf{\textsc{(ii)}}] $\delta(v)=\deg_T(v)-\deg_H(v)\geq 3$.
% \end{itemize}\end{theorem}

% We will show in the NIM enumeration section that NIM trees on $n$ vertices, counted by $N(n)$, are asymptotically vanishing in the class of caterpillars on $n$ vertices, counted by $C(n)$. This will utilize the following key result, a proof of which can be found in \cite{FrankSchwenk}.

% \begin{theorem}\label{caterpillars}
% $C(n)=2^{n-4}+2^{\lfloor (n-4)/2\rfloor}$.
% \end{theorem}

\begin{example}
If $k=1$ then $k$-NIM trees are known as simply \textit{NIM trees}.
% In this case $T$ has $\mathcal{L}(T) = \{(M(T), 1, \dots, 1)\}$ and is simply called NIM.
Figure \ref{fig:nimexample} shows a NIM tree on $49$ vertices.

\begin{figure}[h]
\centering
\includegraphics[scale=0.15]{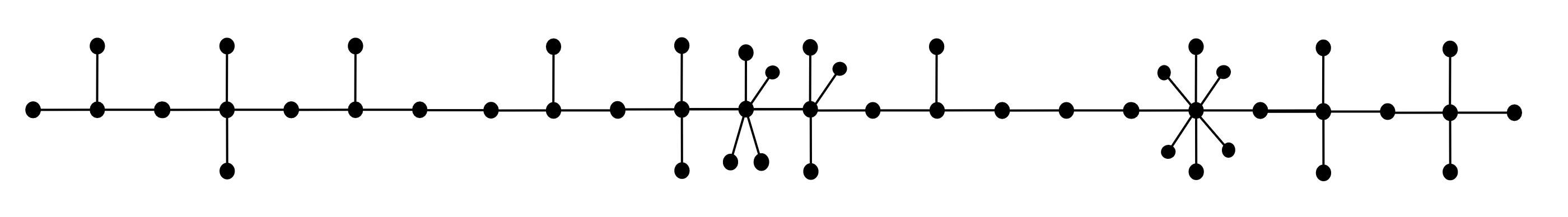}
\caption{An example of a NIM tree on $49$ vertices.}
\label{fig:nimexample}
\end{figure}
\end{example}

Finally, we present some background for use in the NIM enumeration section, which utilizes the symbolic method for constructing ordinary generating functions (OGFs); see \cite{flajolet_sedgewick_2013}.

\begin{definition}[undirected sequence construction]\label{USEQ}
% Given a combinatorial class $\mathcal{C}$, we will soon use the construction for the undirected sequence, which we will denote by $\USeq_{\geq 1}(\mathcal{C})$. 
This is the sequence construction $\Seq$ but counts two elements isomorphic to each other through reflection as the same element, which is what we mean when we say we treat these two objects as \emph{symmetrically equivalent}. We have for two combinatorial classes $\mathcal{B}$ and $\mathcal{C}$ that for $B(\*x)$ and $C(\*x)$ their respective OGFs, the following is known (see e.g. p. 86 of \cite{flajolet_sedgewick_2013} for reference)
\begin{align*}
    \mathcal{B}\cong\USeq_{\geq 1}\left(\mathcal{C}\right)
    \Longrightarrow
    B(\mathbf{x})=\frac{1}{2}\frac{1}{1-C(\mathbf{x})}+\frac{1}{2}\frac{1+C(\mathbf{x})}{1-C(\mathbf{x}^2)}-1.
\end{align*}\end{definition}

\begin{definition}[combinatorial substitution]
Given an object $c$ in some combinatorial class $\mathcal{C}$, if $c$ is marked by variables $\mathcal{Z},\mathcal{X}_1,\dots,\mathcal{X}_k$ then we denote by $c\left[\mathcal{X}_1\mapsto F_1(\mathcal{Z}),\dots,\mathcal{X}_k\mapsto F_k(\mathcal{Z})\right]$ the corresponding element in the new combinatorial class resulting from this substitution, denoted by  $\mathcal{C}\left[\mathcal{X}_1\mapsto F_1(\mathcal{Z}),\dots,\mathcal{X}_k\mapsto F_k(\mathcal{Z})\right]$ (where the $F_i$ map $\mathcal{Z}$ to some some function of $\mathcal{Z}$), the elements of which are only marked by $\mathcal{Z}$.

The OGF for a substitution corresponds with composition of the OGFs. That is, if $C_*(z)$ is the OGF for $\mathcal{C}\left[\mathcal{X}_1\mapsto F_1(\mathcal{Z}),\mathcal{X}_k\mapsto F_k(\mathcal{Z})\right]$, then $C_*$ is a univariate OGF with $C_*(z)=C(z,f_1(z),\dots,f_k(u_k))$, where $f_i(z)$ denotes the OGF for the combinatorial class $F_i(\mathcal{Z})$.  In summary, we have for combinatorial classes $\mathcal{B}$ and $\mathcal{C}$ that 
\begin{equation}
 \mathcal{B}\cong \mathcal{C}\left[\mathcal{X}_1\mapsto F_1(\mathcal{Z}),\dots,\mathcal{X}_k\mapsto F_k(\mathcal{Z})\right]
 \Longrightarrow 
 B(z)=C(z,f_1(z),\dots,f_k(z)).
\end{equation} 
(For more on combinatorial substitution, see Chapters I and III of \cite{flajolet_sedgewick_2013}.)
\end{definition}

\section{Characterization of \texorpdfstring{$k$}{k}-NIM trees}

We first present the main result, a characterization of $k$-NIM trees for all $k\geq 1$.

\begin{theorem}\label{mastertheorem}
Suppose $k\geq 1$ and $T$ is a tree with $M(T)\geq k+1$. Then $T$ is $k$-NIM if and only if for each HDV $v$ of $T$,
\begin{itemize}\label{nimcharacterization}
    \item[\textbf{\textsc{(i)}}] at most $3-k$ components of $T-v$ have more than one vertex, and
    \item[\textbf{\textsc{(ii)}}] $\delta(v)=\deg_T(v)-\deg_H(v)\geq k+1$.
\end{itemize}\end{theorem} 

Theorem 1 of \cite{johnson_saiago_2018} establishes this result for the NIM case (that is, the case $k=1$). The following two subsections verify $k=2$ and $k=3$, respectively, and the final subsection verifies the theorem for all $k\geq 4$, that is, that no $k$-NIM trees exist for $k\geq 4$.

\subsection{2-NIM trees}
In this subsection, we prove the special case of Theorem \ref{mastertheorem} for the case $k=2$. 

% \begin{example}
% By theorems \ref{monograph521} and \ref{mastertheorem}, the tree given in Figure \ref{fig:nimexample} is NIM but not 2-NIM. 
% For a tree that is both 1-NIM and 2-NIM, see Example \ref{firstNIMexample}.
% \end{example}
% We first state and prove an auxiliary lemma.

\begin{lemma}\label{2nim2hdv} If $k=2$ and a tree $T$ satisfies \textbf{\textsc{(i)}}, then $T$ has at most two HDVs.

\begin{proof} 
If $T$ satisfies \textbf{\textsc{(i)}} but has three HDVs then there's an induced path in $T$ containing them. Then one of them, say $v$, is not an end-vertex of this path. But then $T-v$ has at least two branches containing an HDV part of the induced path, which implies $T-v$ has two components with more than one vertex, contradicting \textbf{(i)}.\end{proof}\end{lemma}

\begin{proof}[Proof of Theorem \ref{mastertheorem} for $k=2$]
Fix a tree $T$ with $M(T)\geq 3$. We first show that \textbf{\textsc{(i)}} and \textbf{\textsc{(ii)}} together are sufficient. Let $A\in\mathcal{S}(T)$ have an eigenvalue $\lambda$ of multiplicity $M(T)-1$. Then $\am_A(\lambda)\geq 2$, so by Theorem \ref{monograph423} there is a fragmenting Parter set $V$ for which $\am_{A(V)}(\lambda)=M(T)-1+|V|$. By Lemma \ref{2nim2hdv}, $T$ has either one or two HDVs (not zero, since that would imply $M(T)=1\leq 2$):

\begin{itemize} 
\item \textit{Case 1: $T$ has exactly one HDV, $v$.} By \textbf{\textsc{(ii)}} we know $\deg_T(v)\geq 4$, so $v$ is Parter for $\lambda$ by Theorem \ref{parterwiener}. Then by Theorem \ref{monograph423} there is a nonempty Parter set $V$ such that $\lambda$ is an eigenvalue of $M(T)-1+|V|$ direct summands of $A(V)$ with multiplicity 1 in each. By Theorem \ref{monograph429}, we can choose $V$ to contain only HDVs; $v$ is the only HDV of $T$, so $V=\{v\}$ is a Parter set for $\lambda$. Then $\am_{A(v)}(\lambda)=M(T)-1+|V|=M(T)$, where each of the $\deg_{T}(v)\geq 4$ components $T_i$ of $T-v$ have $\am_{A[T_i]}(\lambda)=1$. 

By \textbf{\textsc{(i)}}, either all or all but one of these components must be singletons. Suppose toward a contradiction $\alpha\neq \lambda$ is also a high multiplicity eigenvalue of $A$. Then, because $v$ is the only HDV in $T$, $v$ must be Parter for both $\alpha$ and $\lambda$, so at least three of the $T_i$ must be assigned $\alpha$ by Theorem \ref{parterwiener}. There is at most one non-singleton component in $T-v$ by \textbf{\textsc{(ii)}}, so by the pigeonhole principle a singleton is assigned both $\alpha$ and $\lambda$. It follows that the assignment is invalid by overloading, so no such $\alpha$ exists, a contradiction.

\item \textit{Case 2: $T$ has exactly two HDVs, $v_1$ and $v_2$.} By \textbf{\textsc{(ii)}} we know $\deg_T(v_i)\geq 4$, $i=1,2$. Suppose toward contradiction $\alpha\neq \lambda$ is also a high multiplicity eigenvalue of $A$. $v_1$ (or $v_2$) alone cannot be Parter for both $\alpha$ and $\lambda$ by the same reasoning as in the previous case, so we can assume without loss of generality that $v_1$ is strong Parter for $\lambda$ and $v_2$ is strong Parter for $\alpha$. Then by Theorem \ref{monograph423} there are nonempty Parter sets $V_\lambda, V_\alpha$ such that $\lambda$ (resp. $\alpha$) is an eigenvalue of $M(T)-1+|V_{\lambda}|$ (resp. $\am_A(\alpha)+|V_{\alpha}|$) direct summands of $A(V_\lambda)$ (resp. of $A(V_\alpha)$),
where $\lambda$ (resp. $\alpha$) has multiplicity $1$ in each. By Theorem \ref{monograph429}, we can choose $V_\lambda$ (resp. $V_\alpha$) to contain only HDVs, so there are only three possible cases:
\begin{itemize} 
\item \textbf{Case 2a}: $V_\lambda=V_{\alpha}=\{v_1,v_2\}$. Then $M(T)-1+|V_\lambda|=M(T)+1$ components of $T-V_\lambda$ are assigned $\lambda$ with multiplicity $1$ in each and $\am_A(\alpha)+|V_\alpha|=\am_A(\alpha)+2$ components of $T-V_\alpha$ are assigned $\alpha$ with multiplicity $1$ in each. Notice that $\{v_1,v_2\}$ is an RPM set for $T$, since by \textbf{\textsc{(ii)}} we have $\delta(v_i)\geq 4\geq 2$, $i=1,2$. It follows that there are $\Delta(T)+|\{v_1,v_2\}|=M(T)+2$ components (all paths) in $T-\{v_1,v_2\}$. All but at most two of these must be singletons by \textbf{\textsc{(i)}}, so at least one singleton must be overloaded. 
This rules out this case.
    
\item \textbf{Case 2b}: $V_\lambda=\{v_1,v_2\}$, $V_\alpha=\{v_2\}$. Then $M(T)-1+|V_\lambda|=M(T)+1$ components of $T-V_\lambda$ are assigned $\lambda$ with multiplicity $1$ in each and $\am_A(\alpha)+2$ components of $T-V_\alpha$ are assigned $\alpha$ with multiplicity $1$ in each. 

There are $\Delta(T)+|\{v_1,v_2\}|=M(T)+2$ components (all paths) in $T-\{v_1,v_2\}$ by the same argument from Case 2a. Moreover, all but at most two of these must be singletons by \textbf{\textsc{(i)}} (in fact, at most one). On the other hand, there are $\deg_T(v_2)$ components in $T-v_2$, so $\am_{A(v_2)}(\alpha)=\am_A(\alpha)+1$, which is $\geq 3$ since $\alpha$ is a high multiplicity eigenvalue for $A$. Thus, since all but (at most) one of the components of $T-V_\alpha=T-v_2$ must be singletons, there must be $\geq 2$ singletons present in both $T-V_\lambda$ and $T-V_\alpha$. If no singleton is to be overloaded, each must be assigned an eigenvalue at most once. Since there are $M(T)+2$ branches formed by the removal of $V_\lambda$, and we must assign $\lambda$ to $M(T)+1$ of them, at least one of these singletons must be assigned $\lambda$ by the pigeonhole principle. But then $3$ branches of $T-V_\alpha$ must be assigned $\alpha$, and there are only two branches that do not cause overloading when assigning---the unassigned singleton, and the non-singleton component of $T-v_2$. So, we must assign $\alpha$ to another singleton by the pigeonhole principle. But we already assigned $\lambda$ to that singleton, so this assignment is invalid by overloading.

\item \textbf{Case 2c}: $V_{\lambda} = \{v_1\}$. We need that $\am_A(\lambda)=M(T)-1$ and that $\am_A(\alpha)\geq 2$. $\am_{A(v_1)}(\lambda)=M(T)$ since $V_\lambda=\{v_1\}$ is a Parter set, so $\lambda$ must be assigned to $M(T)$ branches at $v_1$, which requires $\deg_T(v_1)\geq M(T)$. But we know from Case 2a that $\{v_1,v_2\}$ is an RPM set, so if $\deg_T(v_1)\geq M(T)$ then the following holds:
\begin{align*}
    M(T)+2=\Delta(T)+|\{v_1,v_2\}|
    &=
    \text{The number of components of }(T-\{v_1,v_2\})
    \\
    &\geq \underbrace{(M(T)-1)}_{\text{pendant vertices at $v_1$}}+\underbrace{1}_{\substack{\text{path connecting} \\ {\text{$v_1$ and $v_2$}}}}+\underbrace{\deg_T(v_2)-1}_{\text{pendant vertices at $v_2$}}\\
    &\geq M(T)+3 \tag{since $\deg_T(v_2)\geq 4$ by \textbf{\textsc{(ii)}}},
\end{align*} 
a contradiction. This rules out this case. 
\end{itemize} 

Since the above cases cover all possibilities, \textbf{\textsc{(i)}} and \textbf{\textsc{(ii)}} are together sufficient.

\end{itemize}

We now show that \textbf{\textsc{(i)}} and \textbf{\textsc{(ii)}} are necessary conditions. 
% We do this by showing that if either \textbf{\textsc{(i)}} or \textbf{\textsc{(ii)}} fails for a tree $T$, then there is some $A\in\mathcal{S}(T)$ with the multiplicity list $(M(T)-1,p,\dots)$, $p\geq 2$.
Suppose toward contradiction some tree $T$ is $2$-NIM but that condition \textbf{\textsc{(i)}} fails for $T$ or condition \textbf{\textsc{(ii)}} fails for $T$. 
\begin{itemize} 

\item \emph{Case 1: $T$ does not satisfy \textbf{\textsc{(i)}}}. 
Then there's some HDV $v_*$ such that $T-v_*$ has two non-singleton components, say $T_1$ and $T_2$. Let $A\in\mathcal{S}(T)$ satisfy $\am_A(\lambda)=M(T)-1$ and denote by $V_\lambda$ a fragmenting Parter set for $\lambda$ in $T$ ($V_\lambda$ exists by Theorem \ref{monograph423}). We now split into cases depending on the number of HDVs of $T$.
   
\begin{itemize} 
\item  \textit{Case 1.1: $T$ has a single HDV}. Then $Q = \{v_*\}$ is an RPM set for $T$, and so $\deg_T(v_*) = M(T)+1$. Then since $M(T) \geq 3$, we have $\deg_T(v_*) \geq 4$. Let $f$ be the number of arms of $T$ at $v_*$ (so $f = \deg_T(v_*)$). Since \textbf{\textsc{(i)}} fails for $T$, there are two arms of length $\geq 2$. We will construct $B\in\mathcal{S}(T)$ with multiplicity list $(M(T)-1,2,\dots)$. This will imply $T$ is not $2$-NIM because then there is a list of $\mathcal{L}(T)$ such that a multiplicity of $M(T)-1$ does not imply the simplicity of the remaining eigenvalues, as the list contains an eigenvalue $\mu$ with $\am_B(\mu)=2$. To ensure $\am_B(\mu)= 2$, three components of $T-v_*$ must be assigned $\mu$. We assign both $\lambda$ and $\mu$ to $T_1$ and then assign both $\lambda$ and $\mu$ to $T_2$. We then assign $\mu$ to one of the $f-2$ remaining arms. Finally, assign $\lambda$ to each of the $f-3$ remaining arms. Then $\mu$ is assigned to three branches at $v_*$, so is realizable by Corollary 5.2.3 of \cite{johnson_saiago_2018}. Hence $T$ is not $2$-NIM, a contradiction. 

\item \textit{Case 1.2: $T$ has two or more HDVs}. Then $v_*$ may or may not be Parter for $\lambda$, giving two cases. Our strategy is to construct a near-assignment $\mathcal{A}=\{\mathcal{A}_\lambda,\mathcal{A}_\mu\}$ with fragmenting Parter sets $V_\lambda$, $V_\mu$ corresponding to the list  $(M(T)-1,p,1^{n-(M(T)-1)-p})$ for $p\geq 2$.
% ($V_\lambda$ exists by Theorem \ref{monograph423})
We will then apply Lemma \ref{M2thm2pt5}, which guarantees the existence of $A\in\mathcal{S}(T)$ with each component of $T-V_{\lambda}$ being assigned $\lambda$ and, for $\mu\neq \lambda$, $\mu$ being assigned to $\geq 3$ connected components of $T-V_\alpha$, should we find a corresponding valid near-assignment, which is realizable by Lemma \ref{M2thm2pt5}.

\begin{itemize} 
\item \textit{Case 1.2a: $v_*$ is Parter for $\lambda$ in $T$}. Then either (I) $V_\lambda=\{v_*\}$ or (II) $V_\lambda=\{v_*,v_{i_2},\dots,v_{i_h}\}$.
   
\begin{itemize} 
\item We show that (I) is not possible. Suppose for a contradiction that it was. Then, since $V_\lambda$ is then clearly minimal, we have by (the contrapositive of) Corollary 4.2.8 in \cite{johnson_saiago_2018} that since $k=|V_\lambda|\not >1$, either $V_\lambda$ is not a fragmenting Parter set for $\lambda$ of cardinality $k$ or $v_*\in V_\lambda$ is not a Parter vertex for $\lambda$ of degree less than $\am_A(\lambda)+1=(M(T)-1)+1=M(T)$. Since $V_\lambda$ is by definition/hypothesis the fragmenting Parter set for $\lambda$ and has, in this case, cardinality $k=1$, then $d(v_*) \geq M(T)$. Since condition \textbf{(i)} does not hold, $v_*$ contains at least two components with more than one vertex, and at least one of these components must contain an HDV $w$, and thus be of size at least three. We assign to each branch at $v_*$ the eigenvalue $\lambda$, and to each branch at $w$ the eigenvalue $\alpha$. Since $d(w) \geq 3$, this gives an assignment candidate to the list $(M(T)-1, 2, \dots)$, and no overloading occurs because the component at $v$ containing $w$ is assigned $3$ times, and has size at least $3$. Per Theorem 1.5, It follows that our tree is not $2$-NIM, and thus (I) is impossible. 
   
\item We now consider (II), that $V_\lambda=\{v_*,v_{i_2},\dots,v_{i_h}\}$. Now $T$ has $\geq h$ HDVs, and $T$ has at least two branches at $v_*$ with more than one vertex. Since $V_\lambda$ is by definition a fragmenting Parter set for $\lambda$ in $T$, we have that $\am_{A(V_\lambda)}(\lambda)=(M(T)-1)+|V_\lambda|=(M(T)-1)+h$. Upon deletion of the $h$ HDVs, there must be, therefore, $(M(T)-1)+h$ components that are assigned $\lambda$. Either $V_\lambda$ contains all HDVs or just some. We first deal with the case that $V_\lambda$ contains all HDVs, meaning that $T$ has $h$ HDVs in total. Then by Theorem \ref{monograph429} we can choose an RPM set $Q$ for $T$ that is a subset of $V_{\lambda}$ that allows any RPM set to consist only of HDVs in $T$ since $M(T)-1\geq 2$). Furthermore, we choose $Q \subset V_{\lambda}$ to be the largest of these. Suppose for a contradiction that the fragmenting Parter set $V_\lambda$ is minimal. Then $\am_{A[T-V_\lambda]}(\lambda)=(M(T)-1)+h$ and $\am_{A[T-(V_\lambda\setminus \{w\})]}<M(T)-1+h$ for any HDV $w$ in $T$. Now, any fragmenting Parter set $Q$ can be made such that $Q\subseteq V_{\lambda}$, since $V_\lambda$ contains all the HDVs present in $T$. Moreover, there's a $B\in\mathcal{S}(T)$ that has an eigenvalue $\beta$ of $M(T)$ for which $Q$ is a fragmenting Parter set for $\beta$ in $B$. This means that $T-Q$ leaves at least $\geq M(T)+|Q|$ \textit{paths}. So, since $Q\subseteq V_\lambda$, if we obtain $T-V_\lambda$ by first deleting $Q$ from $T$ and then by deleting $V_\lambda\setminus Q$ from what results, we should have $\geq M(T)+|Q|$ paths. But we need that $M(T)+(h-1)$ direct summands of this are assigned $\lambda$, so we would like to have $h-1\leq Q$. The vertices in $V_\lambda\setminus Q$ must be endpoints of the remaining paths that form the forest $T-Q$, since if they weren't, then they could be included in the RPM set. This is because if they were on the interior of any of the paths in $T-Q$, then any \textit{one} vertex $y\in V_\lambda\setminus Q$ being deleted would increase the number of components by \textit{one}, which would mean that an RPM set $Q'=Q\cup \{y\}$ could be constructed to include them, contradicting that $Q$ is at the maximum size. Thus, deleting the vertices in $V_\lambda\setminus Q$ from $T-Q$ results in either the same number of paths or fewer paths (should any vertex in $V_\lambda\setminus Q$ be a singleton in $T-Q$ so that upon its deletion the number of paths decreases, or also in the case where there is a 2-path in $T-Q$ each vertex is in $V_{\lambda}-Q$). The takeaway here is that $T-V_\lambda$ has no more components than $T-Q$. Since $T-Q$ has at least $M(T)+|Q|$ components and $T-V_\lambda$ has at least $(M(T)-1)+h$ components and $|Q|\leq h$, we have that $|T-Q|\geq |T-V_\lambda|$ implies $M(T)+|Q|\geq (M(T)-1)+h$ implies $|Q|\geq h-1$. But $|Q|\leq h$, so 
\begin{align}\label{h1Qh}
                h-1\leq |Q|\leq h.
\end{align} 
If $|Q|=h$, then $Q=V_\lambda$, which means that there's an RPM set $Q$ for $T$, containing $v_*$, and such that $\am_{A(Q)}=\am_A(\lambda)+|Q|$ since $V_\lambda$ by definition is a fragmenting Parter set for $\lambda$ (and thus now so is $Q$). But Lemma \ref{monograph425} implies that $v_*$ is Parter for $\lambda$ but that $\lambda$ is of multiplicity $M(T)$, a contradiction. It then follows from \eqref{h1Qh} that $|Q|=h-1$, so there's a vertex in $V_\lambda$ that is not in $Q$, meaning $Q$ is a \emph{proper} subset of $V_\lambda$. Now, if $|Q|=h-1$, then $T-Q$ has at least $M(T)+|Q|=(M(T)-1)+h$ components. But we assumed that $V_\lambda$ was minimal, so no set $P$ of Parter vertices in $T$  that is smaller than $V_\lambda$ can exist such that $\am_{A(P)}(\lambda)=(M(T)-1)+h$. But $Q$ demonstrates that this is indeed the case, a contradiction. It then follows that $V_{\lambda}$ is not minimal. Therefore, we can choose any HDV $z\in V_\lambda$ and we will have that $V_{\lambda}':= V_{\lambda} \setminus \{z\}$ is also a fragmenting Parter set for $T$. We now need to show that we can make $z$ strong Parter for $\mu$ in an assignment without overloading. We note that the branches at $z$ in $T-V_\lambda'$ are all branches different from those of $T-z$ when $z$ has two pendant paths, and let $z$ be such a vertex as guaranteed by Lemma 3.3.3 in $\cite{johnson_saiago_2018}$. $z$ must have another branch, so we have now established that $T-z$ consists of $\geq 3$ branches, two pendant paths that are not assigned $\lambda$ since we can choose $z$ to be the one HDV that satisfies $z\not\in V_\lambda$, and also the third branch that is guaranteed, to at least connect $z$ and its two pendant paths to the rest of the tree, is also not assigned $\lambda$. This means that at least three components of $T-z$ are not assigned $\lambda$, so we can assign them $\mu$ without overloading. We now set $\mathcal{A}_\mu$ to be the set of some three of these components in $T-z$,  $V_\mu=\{z\}$, and then set $\mathcal{A}_\lambda$ to be the set of components of $T-V_{\lambda}'= T- (V_\lambda \setminus\{v_*\})$. Since this is a valid near-assignment per our remarks throughout, we're done by our initial remarks.
\end{itemize}
            
On the other hand, if $V_\lambda$ does not contain all the HDVs, then we can choose an HDV $y$ in $T$ with $y\not\in V_{\lambda}$ for which $T-y$ has $\geq 3$ components, none of which are components of $T-V_{\lambda}$. Such a $y$ exists---indeed, if $Q$ is an RPM set for $T$ and $Q\subseteq V_{\lambda}$ then we're done by the above case; if instead $V_{\lambda}\subseteq Q$, then then by Lemma 5.1.5 in $\cite{johnson_saiago_2018}$ none of the $\geq 3$ components of $T-y$ 
% is a component of $T-V_\lambda$
are assigned $\lambda$, so for three branches at $y$, say $T_1,T_2,T_3$, we can set $\mathcal{A}_\alpha=\{T_1,T_2,T_3\}$, with $V_\alpha=\{y\}$.
        
\item \textit{Case 1.2b: $v_*$ is not Parter for $\lambda$ in $T$}.
% $v_*\not\in V_\lambda$}.
% , i.e. no fragmenting Parter set exists for $\lambda$ in $T$ that contains $v_*$
% Then $v_*$ is not Parter. 
Observe that everything in $V_\lambda$ has at most one non-pendant neighbor, since if there were a vertex $u$ in $V_\lambda$ with $\geq$ two non-pendant neighbors, then we can choose $v_*=u$.

We're done if there is a vertex $y$ of $T-V_\lambda$ such that $\deg_{T_y}(y)\geq 3$, where $T_y$ is the component of $T-V_\lambda$ containing $y$. Indeed, if this were the case then we can set $\mathcal{A_\mu}=\{X_i:X_i\text{ is a component of } T_y-y,i=1,2,3\}$. Suppose for a contradiction that $T-V_\lambda$ contains nothing more than paths and singletons. $V_\lambda$ satisfies the hypothesis for the set $Q$ in Lemma 3.3.7 of \cite{johnson_saiago_2018}, so $V_\lambda$ is an RPM set for $T$. Then by Lemma \ref{monograph425} $\lambda$ is an eigenvalue of multiplicity $M(T)$, not $M(T)-1$, a contradiction. Thus $T-V_\lambda$ must contain the desired vertex $y$, so we're done by our initial remarks. 
% and per our remarks we set $V_\mu:=\{y\}, \mathcal{A}_\mu:=\{X:X_i\in T_y-y,i=1,2,3\}$. Then we're done by our initial remarks. 
\end{itemize}
We conclude that \textbf{\textsc{(i)}} is a necessary condition.
\end{itemize}

\item\textit{Case 2: $T$ satisfies \textbf{\textsc{(i)}} (it must by the above argument), but not \textbf{\textsc{(ii)}}}. Then there is an HDV $v$ with $\leq 3$ non-HDV neighbors. $\deg_H(v) \leq 1$ by \textbf{\textsc{(i)}} and \ref{2nim2hdv}, so since $v$ is an HDV there are two possibilities: 

\begin{itemize} 
\item \textit{Case 2(a): $\deg_T(v) = \deg_H(v) + 3$}. Then necessarily $\deg_T(v) = 3$ and $\deg_H(v) = 1$, since $\deg_H(v) \leq 1$. Then we know that $v$ is adjacent to another HDV $v'$, and since there can be at most two HDVs by Lemma \ref{2nim2hdv} we know that $T$ takes the form of a double star. Furthermore, we know that $Q = \{v, v'\}$ forms an RPM set for $T$ since $\delta(v),\delta(v') \geq 2$. We know that there are two pendant vertices at $v$. We have that $P(T)=|T-Q|-2 = \deg_T(v) + \deg_T(v') - 3$, since we overcount the edge connecting the two HDVs, and since $\deg_T(v)=3$. By Lemma \ref{monograph425} we therefore have that there is a matrix $A \in \mathcal{S}(T)$ that has an eigenvalue of multiplicity $M(T)=\deg_T(v')$. Our Parter set for $\lambda$ (the eigenvalue of multiplicity $M(T) - 1$) is $\{v'\}$ and our Parter set for $\alpha$, which must occur with multiplicity is $2$, is $\{v\}$. Then construct the matrix $A \in \mathcal{S}(T)$ by prescribing the following: $A - v'$ leaves a block diagonal matrix, with all but one block (of size 3-by-3) being singletons, and to each such block assign $\lambda$ as an eigenvalue. This is certainly possible as removing $v'$ leaves a bunch of singletons and this 3-path, each of which can be assigned such a value. This assigns $\deg_T(v')$ blocks $\lambda$, and in the re-addition of $v'$ has $\lambda$ occurring of multiplicity $M(T)-1$, I think, that is if $v'$ was Parter for $\lambda$. 
    
Then consider $A[T-v]$. We want to assign each block the eigenvalue $\alpha$ so as not to cause overloading. Since the large component of $T-v$ is not fully assigned out, it does not have $\lambda$ occur $\deg_T(v')-1$ times, rather less. Then we assign $\alpha$ to this component and the two pendant vertices of $v$. This gives a valid assignment, and Lemma \ref{M2thm2pt5} gives a matrix $A$, which gives a multiplicity list $(M(T)-1, 2, \dots)$.
    
\item \textit{Case 2(b): $\deg_T(v) = \deg_H(v) + 2$}. 
% In other words, $v$ has $\deg_T(v) = \deg_H(v) + 3$ ($\deg T(v) = 3$ or $4$) and $\deg_H(v)$ is $0$ or $1$). 
By case 2, we can assume that $\deg_T(v) \geq \deg_H(v) + 3$ for all HDVs $v$, and for the particular $v$ that has equality. We have two cases, either $\deg_T(v) = 4$ and $\deg_H(v) = 1$ or $\deg_T(v) = 3$ and $\deg_H(v) = 0$.
\end{itemize}

\begin{itemize}
\item \textit{Case 2(b)(i): $\deg_T(v) = 4$ and $\deg_H(v) = 1$}. In this case, another pendant vertex is attached to $v$, and as a result, the path cover number goes up by $1$, with reference to Case 1. To assign $M(T)$ subtrees $\lambda$, we follow the assignment described in Case 1, and in the new pendant vertex at $v$, we assign $\lambda$. In this way, we have assigned $M(T)$ subtrees $\lambda$ and $3$ subtrees $\alpha$. No overloading occurs as there was none in Case 1, and none can result because of assigning $\lambda$ to a single pendant vertex.
     
\item \textit{Case 2(b)(ii): $\deg_T(v) = 3$ and $\deg_H(v) = 0$}. Either there is a second HDV or this is a 3-star, and there's only one interesting case. In this case, the difference from Case 2 is that the two HDVs are not joined together by an edge---rather, by a path. The path cover number does not change, and the same assignment from the first case will do, where we include the path in each subtree where necessary. 
% Again no overloading can occur because none did in Case 1.
\end{itemize}

\end{itemize}
This proves necessity of \textbf{\textsc{(i)}} and \textbf{\textsc{(ii)}}, completing the proof of Theorem \ref{mastertheorem} in the case $k=2$. 
\end{proof}

The following is an immediate consequence of Theorem \ref{mastertheorem} in the case $k=2$. 

\begin{corollary} 
If $T$ is $2$-NIM, then $T$ is NIM.
\end{corollary}

%We can give a graph-theoretic description of all types of $2$-NIM trees $T$ using the characterization given by Theorem \ref{2nimcharacterization}.
% \vspace{-0.2in}\begin{figure}[h]
%     \hspace{-1in}\includegraphics[scale=0.25]{dtsra.png}
%     \vspace{-0.35in}\caption{If $T$ has two HDVs then $T$ must be of the above form.}
%     \label{fig:2nim2hdvexample}
% \end{figure}

%If there is only one HDV, then $T$ is either a star or a generalized star, and if there are no HDVs we have the path. The second condition places additional restrictions on the minimum number of pendants on the HDVs in each case. One more observation as a result of this characterization is that no tree with $M(T) = 2$ can be $2$-NIM, because $2-1 = 1$ occurring in a list does not imply the rest of the list is all $1$s (for example, $(2,1,\dots,1)$ clearly occurs as a list here).

\subsection{3-NIM Trees}
In this subsection, we prove the statement of Theorem \ref{mastertheorem} in the case $k=3$. This will follow from the following theorem.
% Recall that the simple star $S_n$ is the tree on $n$ vertices with one HDV at the center (the central vertex) and the remaining vertices pendant to the central vertex.

\begin{theorem} \label{3nimsimplestar}
Let $T$ be a tree on $n$ vertices with $M(T)=4$. Then $T$ is $3$-NIM if and only if $T=S_n$.

\begin{proof}
We first show that if $T$ is $3$-NIM then it must be the simple star $S_n$. We have that $T$ is 3-NIM, so $(M(T)-2,(1)^{n-M(T)+2})$ is the only list in $\mathcal{L}(T)$ containing the multiplicity $M(T)-2$. The base cases $M(T)\leq 3$ have already been checked and are straightforward calculations, so we suppose $M(T)\geq 4$. Then $M(T)-2\geq 2$, so by Theorem \ref{parterwiener} there's a strong Parter HDV $v$ for $\lambda$ in $T$. Now, suppose toward contradiction $T$ is not a simple star. Then $T$ is either a generalized star or has more than one HDV ($T$ being a path says nothing about our claim since then $\mathcal{L}(T)=(1^n)$, so we may assume that $T$ is not a path):
\begin{itemize}
\item
\textit{Case 1: $T$ is a (non-path) \textit{generalized} star}. Now, one arm $\ell_k$ of $v$ must have length $\ell_k\geq 2$. Even if all the remaining branches at $v$ are pendant vertices, we can assign to $\deg(v)-3$ of these branches the eigenvalue $\alpha$ with $\am_A(\alpha)=M(T)-2$, and assign to the remaining three components the eigenvalue $\beta$. Consider a candidate-assignment $\mathcal{A}$ corresponding two eigenvalues $\alpha\neq \beta$ with $\am_A(\alpha)=M(T)-2$ and $\am_A(\beta)=2$ where $\mathcal{A}_\alpha$ contains only $\deg(v)-5=M(T)-2$ of the arms arms and where the subtrees (resp. Parter vertices of $T$) assigned to $\alpha$ and $\beta$ are $\mathcal{A}_\alpha$ and $\mathcal{A}_\beta$ (resp. $V_\alpha$, $V_\beta$, respectively. that has assignments  $\ell_1,\dots,\ell_M(T)-2$. $M(T)=n-2=\deg(v)-1$, so $M(T)-2=\deg(v)-3$. Thus three arms can be assigned to the eigenvalue $\beta$ while the rest are assigned $\alpha$, so no overloading occurs.
 
\item 
\textit{Case 2: $T$ has more than one HDV}. Now $T$ has two peripheral HDVs, say $v_1, v_2$, which then
% by the Johnson-Saiago algorithm
% they
belong to an RPM set $Q$. $T-Q$ has at least $M(T) + |Q|$ direct summands, of which we would like to assign $M(T)+|Q|-2$ $\lambda_1$, so we do this for all but two of the components emanating from $v_1$. For $\lambda_2$ let $\{v_1\}$ be the fragmenting Parter set, and assign $\lambda_2$ to the two unassigned components and the subtree containing all other HDVs (there must be at most one such subtree by peripherality). No overloading occurs in this assignment since there is at least one other Parter vertex in this subtree. Thus the assignment is valid, so $(M(T)-1, 2,\dots)$ is realizable.
\end{itemize}

We now show that all simple stars are $3$-NIM. Consider the simple star $S_n$, $n\geq 3$ which has $P(T) = n-2$, so $M(T) = n-2$ (because in a path cover we can create one path between two pendant vertices and the center vertex, forcing the rest of the $n-3$ pendant vertices to be their own path). To show that $S_n$ is $3$-NIM, we only need to show that any list of the form $(M(T)-2,k,\dots)$ for $k\geq 2$ is not in $\mathcal{L}(T)$. Suppose for a contradiction that such a list did indeed exist in $\mathcal{L}(T)$, and let $\lambda$ be the eigenvalue of $A$ with multiplicity $M(T)-2=n-4$. Then, by Theorem \ref{parterwiener}, we have that the central vertex $v$ is Parter for $\lambda$, and thus $\am_{A(v)}(\lambda)=M(T)-2+1=M(T)-1=n-3$, so it follows that $n-3$ of the Pendant vertices must be assigned $\lambda$ as an eigenvalue since the $n-1$ singletons are all that remain when deleting $v$. But $k\geq 2$, so if $\alpha$ is the eigenvalue of multiplicity $k\geq 2$, then again by \ref{parterwiener} we have that $v$ is strong Parter for $\alpha$, so three branches at $v$ must be assigned $\alpha$, but there are only two remaining spots, so by the pigeonhole principle we have that at least one of the pendant vertices must be assigned $\lambda$ and $\alpha$, a contradiction.
% , which is a contradiction to the overloading.
\end{proof}
\end{theorem}

% As a result, we obtain the following as a corollary.

\begin{corollary}\label{3nimcharacterization} Let $T$ be a tree and $M(T)\geq 4$. $T$ is $3$-NIM if and only if for each HDV $v$ of $T$, \begin{itemize}
    \item[\textbf{\textsc{(i)}}] no components of $T-v$ has more than one vertex, and
    \item[\textbf{\textsc{(ii)}}] $\delta(v)=\deg_T(v)-\deg_H(v)\geq 5$.
\end{itemize}\end{corollary}
\begin{proof}
Let $T$ be $S_n$ ($n \geq 5$). It is easy to see that when the HDV of $S_n$ is removed, all components are singletons and that because our HDV has at least degree 5, we'll have $\deg_T(v)-\deg_H(v)\geq \deg_T(v)-0 \geq 5$. For the backward direction, we take a tree that satisfies \textbf{(i)} and \textbf{(ii)}. This tree must have central path length of at most 1 since \textbf{(i)} fails otherwise. Next, all HDVs must have at least degree 5 to satisfy \textbf{(ii)}. The only trees that do both are simple stars $S_n$ for $n \geq 5$. Thus, \textbf{(i)} and \textbf{(ii)} imply that a tree is 3-NIM.  
\end{proof}

\subsection{Nonexistence of \texorpdfstring{$k$}{k}-NIM trees for \texorpdfstring{$k\geq 4$}{k geq 4}}
In this subsection, we verify Theorem \ref{mastertheorem} for $k\geq 4$ by showing that no $k$-NIM trees exist for $k\geq 4$.

\begin{lemma}\label{existenceofperipheralHDVs} Every RPM set for graphs with $M(T)\geq 5$ has a \textit{peripheral HDV}, that is, an HDV of $T$ at which there is a branch containing all other HDVs.
% an HDV with $\geq 2$ pendant paths as branches.

\begin{proof}

Let $Q$ be an RPM set. Let $P$ be a path in $T$ with the most vertices from $Q$ possible. Either of the two HDVs closest to the end-vertices of the path is peripheral since if two of its branches contained an HDV, then the path could be extended to one containing more HDVs.
\end{proof}
% Let $Q$ be the RPM set. Let $P$ be a path in $T$ with the highest possible number of HDVs from the RPM set on it. Pick either of the two HDVs closest to the end of the path, and call it $v$. $v$ is peripheral, for if it weren't, then it contains another HDV from $Q$ in another branch that is not part of the path, but then the path could be augmented to include that HDV and create a longer path with more HDVs from Q. 

% In particular, since $v$ must be an HDV, it has at least two pendant paths, for if not, by the same reasoning, there is an HDV in another branch not part of the path, which we could include to form a new path with more HDVs.
\end{lemma}

\begin{theorem}\label{notreehigherthan4nim}
 No tree is $k$-NIM for $k \geq 4$.
\begin{proof}
We have two cases:
\begin{itemize}
\item First, suppose $T$ is a generalized star. Then the only RPM set is the singleton $\{v\}$, and its removal leaves $\deg_T(v)$ paths, so that $M(T) = \deg_T(v)-1$. We assign $M(T)-2 = d(v)-3$ of those paths $\lambda_1$, and $3$ of the remaining paths $\lambda_2$, in total assigning all paths an eigenvalue without any overloading, yielding a valid assignment. Then by Corollary \ref{alltwoeigassignmentsarerealizable} this assignment is realized by some matrix $A\in\mathcal{S}(T)$.

\item Now suppose instead $T$ has more than one HDV. By the proof of Lemma \ref{existenceofperipheralHDVs}, $T$ has some two peripheral HDVs $v_1$ and $v_2$, each with at least two pendant paths. Then $v_1$ and $v_2$ must be in an RPM set $Q$ for $T$. Letting $v$ be one of $v_1$ or $v_2$, we now construct the following assignment for $(M(T)-3, 2, 1^{n - M(T)-5})$, with $\lambda_1, \lambda_2$ corresponding to the two high multiplicities. $T-Q$ leaves $M(T)+|Q|$ paths, and to $M(T)+|Q|-3$ of them we assign $\lambda_1$ in such a way that $\lambda_1$ is not assigned to at least two of the pendant paths at $v$. Denote by $\mathcal{A}_1$ these $M(T)+|Q|-3$ chosen paths, and let $V_1 = Q$. 

Next, set $V_2 = \{v\}$ and assign $\lambda_2$ to two of the non-assigned branches of $T-v$, as well as the entire subtree containing all other vertices of $Q$ (which is well-defined since $v$ is peripheral). Denote by this assignment $\mathcal{A}_2$. No overloading can occur, since the removal of the other peripheral HDV for $V_1$ allows another degree of choice for us to insert $\lambda_2$ as prescribed by the overloading formula. No other subtrees have any potential to be overloaded because $V_1, V_2$ partition the remainder of the tree into paths. We have then given a valid assignment, so by Lemma \ref{alltwoeigassignmentsarerealizable} it is realizable.
\end{itemize}

In both of the above cases, we can find a matrix $A$ which realizes the multiplicity list $(M(T)-3, 2, 1^{n - M(T)-5})$, and thus $T$ is not $4$-NIM.
\end{proof}
\end{theorem}

\section{Enumeration of \texorpdfstring{$k$}{k}-NIM trees}

In this section we enumerate $k$-NIM trees for all $k\geq 1$. The case $k=4$ is easy---by Theorem \ref{notreehigherthan4nim} there is nothing to count. Additionally, by Theorem \ref{3nimsimplestar}, there is exactly one $3$-NIM tree on $n$ vertices for all $n\geq 5$, namely $S_n$, and no trees are $3$-NIM for $n\leq 5$. Thus, only the $k=1,2$ cases remain. 

\subsection{Enumeration of 2-NIM trees}
In this subsection we enumerate $2$-NIM trees on $n\geq 1$ vertices. Though still nontrivial, Lemma \ref{2nim2hdv} imposes significant structure on $2$-NIM trees. The $2$-NIM sequence, $(N_2(n))_{n\geq 1}$, goes as follows:
\begin{equation*}
    (0,0,0,0,1,2,3,4,6,9,12,16,20,25,30,36,42,49,56,\dots)
\end{equation*} 
\begin{theorem}
For $n \geq 8$, the number of $2$-NIM trees on $n$ vertices is given by
\begin{equation*}N_2(n)=\left\lfloor\frac{n-4}{2}\right\rfloor\left\lceil\frac{n-4}{2}\right\rceil\end{equation*}
\end{theorem}
\begin{proof}
By the Theorem \ref{mastertheorem}, a 2-NIM tree $T$ on $n$ vertices must take one of the following forms:

\hspace{0.35in}
\begin{figure}[h]
\vspace{-0.3in}
\begin{subfigure}{0.5\textwidth}
% \vspace{-0.25in}
\hspace{-0.8in}
\includegraphics[scale = 0.15]{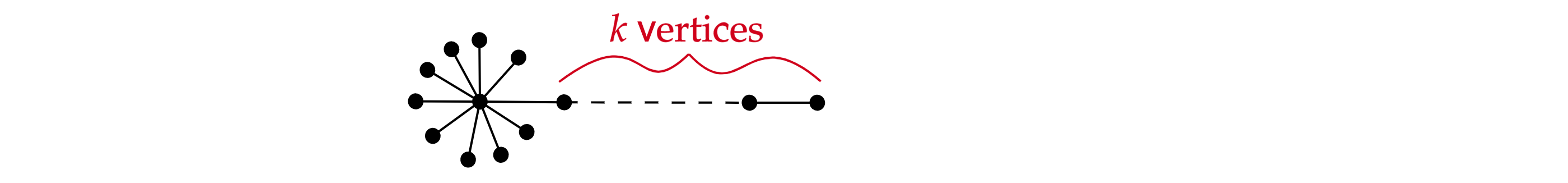}
\end{subfigure}
\hspace{-1in}
\begin{subfigure}{0.5\textwidth}
\includegraphics[scale = 0.15]{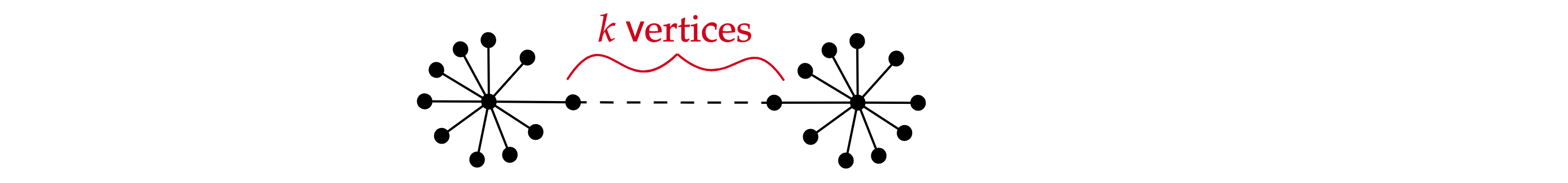} 
\end{subfigure}
\vspace{-0.15in}
\end{figure}
\vspace{-0.05in}

\noindent Straightforward computation yields the result for $n=8,9$. For $n \geq 10$, counting the number of such trees with $n$ vertices yields the formula $n-4 +\left\lfloor \frac{n-8}{2} \right\rfloor +\sum_{k = 1}^{n-8} \left\lfloor \frac{n - 6 - k}{2} \right\rfloor=\left\lfloor\frac{n-4}{2}\right\rfloor\left\lceil\frac{n-4}{2}\right\rceil$.
\end{proof}

\subsection{Enumeration of NIM Trees}

In this subsection, the characterization of NIM (1-NIM) trees in Theorem \ref{nimcharacterization} is be used to enumerate the NIM trees on $n$ vertices, which we denote by $N(n)$. See Figure \ref{fig:nimTable} for a list of computed values. 

This section assumes knowledge of the symbolic method for ordinary generating functions (OGFs). (See Chapters I and III of \cite{flajolet_sedgewick_2013}.)

\begin{figure}[h]
\begin{center}
\begin{tabular}{||l|l||l|l||l|l||l|l||}
\hline
$n$ & $N(n)$ & $n$ & $N(n)$ & $n$ & $N(n)$& $n$ & $N(n)$\\
\hline\hline
 1  & 1    & 14   & 433      & 27 & 1044986        & 40 & 2629706159         \\
 2  & 1    & 15   & 778      & 28 & 1908527        &41 & 4803167836\\
 3  & 1    & 16   & 1416     & 29 & 3485092        & 42 & 8773051285\\
 4  & 2    & 17   & 2564     & 30 & 6365294        & 43 & 16024058213 \\
 5  & 3    & 18   & 4676     & 31 & 11624741       & 44 & 29268175856\\
 6  & 5    & 19   & 8498     & 32 & 21232255       & 45 & 53458678277\\
 7  & 8    & 20   & 15507    & 33 & 38778177       & 46 & 97643083938\\
 8  & 14   & 21   & 28246    & 34 & 70828006       & 47 & 178346438250 \\
 9  & 24   & 22   & 51568    & 35 & 129363233      & 48 & 325752518305\\
 10 & 43   & 23   & 94049    & 36 & 236282260      & 49 & 594991757076\\
 11 & 74   & 24   & 171734   & 37 & 431563697      & 50 & 1086761711574\\
 12 & 134  & 25   & 313417   & 38 & 788254745      & 51 & 1984986757597        \\
 13 & 238  & 26   & 572377   & 39 & 1439742242     & 52 & 3625609406618              \\
\hline
\end{tabular}
\end{center}\caption{Values of the NIM sequence up to $n=52$. This sequence can now be found at OEIS \href{https://oeis.org/A347018}{\textbf{A347018}.}}
\label{fig:nimTable}\end{figure}

\begin{definition}[variable HDV] Given a tree $T$, a \emph{variable HDV} in $T$ is a vertex $v$ of $T$ such that $\deg_T(v)\geq 4$. These will sometimes be marked red in figures involving atoms and skeletons, which are defined below. These are marked by $\mathcal{U}$.
\end{definition}

\begin{definition}[bridge/bridge point] Given a tree $T$ on a central path (such as a caterpillar or NIM tree), a \emph{bridge} in $T$ is a path of length $\geq 1$ induced in $T$ by the vertices with $\deg_T(v)\leq 2$ on the central path of $T$. The length of a bridge is defined by its number of edges. A \emph{bridge point} is a bridge of length $0$ (so a single vertex vertex on the central path with $\deg_T(v)=2$ but with $\delta(v)=0$). These will sometimes be marked blue in figures involving atoms and skeletons, which are defined below. These will soon be marked by $\mathcal{J}$, but initially, we let $\mathcal{J}$ denote the number of components.
\end{definition}

\begin{definition}[atoms] An \emph{atom} is any element of the combinatorial class $\mathcal{A}$, defined in Figure \ref{fig:atoms} below. In particular, the $i$th smallest atom for $i\geq 3$ is marked by $\mathcal{Z}^{4i}\times\mathcal{U}^i\times\mathcal{J}$, with $\mathcal{Z}$ marking vertices.
\end{definition} 

\begin{figure}[h]
    \centering
    \hspace{-0.1in}\includegraphics[scale=0.15]{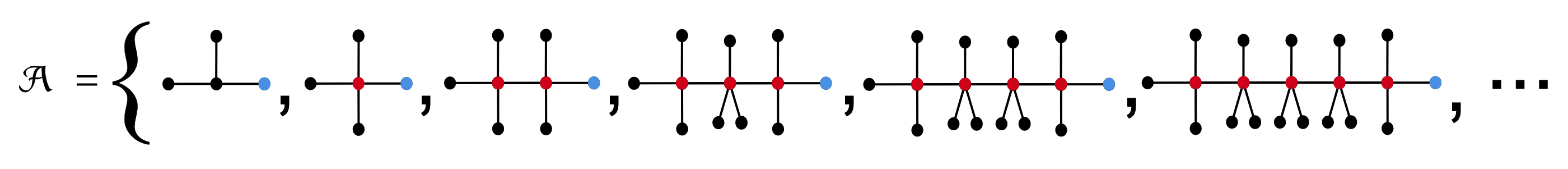}\vspace{-0.1in}
    \caption{The atoms $\mathcal{A}$. The larger atoms follow the obvious pattern.}
    \label{fig:atoms}
\end{figure}

\begin{definition}[molecule] A \emph{molecule} is an undirected sequence\footnote{The use of the term ``undirected sequence'' here indicates that two molecules should be considered distinct only if they are not isomorphic to each other under reflection about the center of their central paths.} of atoms that results from vertex-concatenating atoms to each other at the ends of the central paths. (This is well-defined since each atom is symmetric.)
\end{definition}

\begin{definition}[descendant of a molecule] A tree $T$ is a \emph{descendant} of a molecule $T^*$ if $T$ can be built from $T^*$ by adding bridges at bridge points and pendants (possibly none) at variable HDVs. 

The \emph{descendant class} of $T^*$, denoted $D(T^*)$, is the set of all $T$ that are descendants of $T^*$. We consider $T^*$ to be its own descendant.
\end{definition}

\begin{example} The following figure shows the molecule $T^*$ (left) marked by $\mathcal{Z}^{15}\times \mathcal{U}^3\times\mathcal{J}^3$ (composed of the atoms $\mathcal{Z}^{12}\times\mathcal{U}^3\times\mathcal{J}$ and $\mathcal{Z}^3\times \mathcal{J}$) and one of its descendants $T\in D(T^*)$ (right) generated by the following substitution at variable HDVs and bridge points of the molecule $T^*$:
\begin{equation*}
    T=T^*[\mathcal{U}^3\mapsto (\mathcal{Z},\mathcal{Z},\mathcal{Z}^4), \mathcal{J}^3\mapsto(\* 1,\mathcal{Z}^4,\mathcal{Z}^2)].
\end{equation*} 
\vspace{-0.35in}
\begin{figure}[h]
    \centering
    \vspace{-0.125in}
    \includegraphics[scale=0.15]{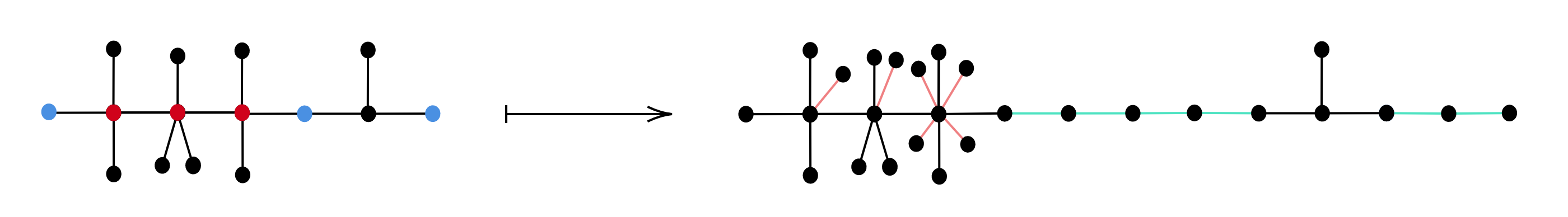}
    \vspace{-0.1in}\caption{From $T^*$ (left), we construct $T\in D(T^*)$ (right) determined by $\mathcal{Z}^{15}\times\left(\mathcal{Z},\mathcal{Z},\mathcal{Z}^4\right)\times\left(\*1,\mathcal{Z}^4,\mathcal{Z}^2\right)$.}
    \label{fig:descendantconstruction}
\end{figure}
\end{example}
\subsubsection{Construction of NIM Trees}

Since there is one path on each $n\geq 1$, we can simply add one at the end after enumerating the non-path NIM trees. We will set aside the paths, which are trivially NIM trees, until the end of the enumeration. Thus, to ease the enumeration, \textit{NIM trees will refer to all non-path trees with the NIM property unless otherwise stated}.

\begin{definition}[NIM skeleton/skeleton class]
Given a NIM Tree $T$, let the \emph{NIM skeleton} of $T$ be the tree $T^*$ that results from an application of Algorithm \ref{nimskeletonalg} below. Define the \emph{skeleton class} of $T$, denoted $\text{Sk}(T^*)$, as the set of all NIM trees $T$ whose NIM skeleton is $T^*$.
\end{definition}

\begin{algorithm}\label{nimskeletonalg} The following algorithm gives the NIM skeleton of a NIM tree $T$.
\begin{itemize}
    \item[1.] Edge contract every pair of adjacent  degree $2$ vertices until no such pairs remain. 
    
    \item[2.] Every induced path in the subgraph of $T$ induced by the (normal) HDVs of $T$ corresponds to the atom in $\mathcal{A}$ with the same number of \emph{variable} HDVs on its central path (possibly $0$ as in the case of the smallest atom). Replace each of these components with their corresponding atom.
\end{itemize}
The resulting tree, denoted by $T^*$, is defined to be the NIM skeleton of $T$.
\end{algorithm}

The following propositions should be clear from Algorithm \ref{nimskeletonalg}:

\begin{proposition} Let $T$ be a NIM tree. Then there is a unique skeleton $T^*$ such that $T^*$ results from applying Algorithm \ref{nimskeletonalg} to $T$.
\label{nimhasuniquenimskeleton}
\end{proposition}
\begin{proposition} $T^*$ is a NIM skeleton if and only if $T^*$ is a molecule.\label{singlenimskeletoniffdescendant}

% \begin{proof}
%     ($\Longrightarrow$) Let $T^*$ be a NIM skeleton. Then $T^*$ results from applying step $2$ to some $T$. But Applying step $2$ to some $T$ gives nothing more than a vertex-concatenation of atoms from $\mathcal{A}$, which is defined to be a molecule, so we're done.

%     ($\Longleftarrow$) Let $T^*$ be a molecule. Molecules are NIM by Theorem \ref{nimcharacterization}, so $T^*$ is NIM. Applying Algorithm \ref{nimskeletonalg} to $T^*$ returns $T^*$: Indeed, step 1 returns the same graph, and step 2 again preserves $T^*$ by definition of a molecule being composed of atoms from $\mathcal{A}$ with vertex-concatenation. Since this is how a NIM skeleton is defined, $T^*$ is a NIM skeleton.
% \end{proof}

\end{proposition}

% The following proposition should also be clear, and we omit its proof.
\begin{proposition} 
Let $T^*$ be a molecule (equivalently, a NIM skeleton). Then $T\in \textnormal{Sk}(T^*)$ if and only if $T\in D(T^*)$.\label{nimskeletoniffdescendant}

% \begin{proof}
%     For the forward direction, let $T^*$ be a NIM skeleton. Let $T\in \text{Sk}(T^*)$, i.e. let $T$ be such that applying Algorithm \ref{nimskeletonalg} to $T$ results in $T^*$. Then $T\in D(T^*)$ is justified by the fact that we can start with $T^*$ and simply reverse the algorithm employed to get $T^*$ from $T$, edge subdividing the appropriate amount of times at bridge points and adding the appropriate amount of pendants at variable HDVs to generate $T$, meaning $T\in D(T^*)$.
    
%     Conversely, let $T^*$ be a molecule. Let $T\in D(T^*)$, i.e. let $T$ be built starting from $T^*$ and adding bridges at its bridge points and pendants at its variable HDVs. Applying Algorithm \ref{nimskeletonalg} to $T$, we have since step $1$ deletes bridges and step $2$ deletes surplus pendants at variable HDVs such that they leave atoms that $T^*$ is the NIM skeleton of $T$, thus $T\in \text{Sk}(T^*)$. 
% \end{proof}

\end{proposition}

Denote by $\mathcal{N}_n$ the combinatorial class of NIM trees on $n$ vertices.
\begin{proposition} 
% We have
As combinatorial classes, 
$\mathcal{N}_n\cong\bigcup_{\substack{\textnormal{$T^*$ is a}\\ \textnormal{molecule}}}D_n(T^*)$.
\end{proposition}

\begin{proof} We first show $D_n(T^*)\subset \mathcal{N}_n$. That is, given a molecule/NIM skeleton $T^*$, each of its descendants is NIM: Let $T^*$ be a NIM skeleton. Let $v$ be an HDV of $T^*$. We have $\delta(v)\geq 3$ for all HDVs since this is true in the atoms and is preserved under vertex-concatenation of atoms at their end-vertices by adding pendants to variable HDVs $v$, and by adding bridges to $T^*$. $T^*$ is NIM, so $T^*-v$ has at most two connected components for each HDV $v$ of $T$ by Theorem \ref{nimcharacterization}. Adding pendants preserves this fact since pendants are not their own connected components. Adding bridges preserves this as well since this only adds vertices to one or the other component that already has more than one vertex unless $v$ is a peripheral HDV of $T^*$. If $v$ is a peripheral HDV in $T^*$, then since $v$ is on the central path and peripheral, $T^*-v$ must have only one  component with more than one vertex, so adding a bridge at any of the two end junctions would only make it such that $T^*-v$ has two components with $\geq 1$ vertex, which Theorem \ref{nimcharacterization} permits. Hence $D_n(T^*)\subset\mathcal{N}_n$.

We now show that $\mathcal{N}_n\subset D_n(T^*)$. That is, given a NIM tree $T$ on $n$ vertices, it is the descendant of a unique molecule $T^*$: Let $T$ be a NIM tree. $T$ has a unique NIM skeleton $T^*$ by Proposition \ref{nimhasuniquenimskeleton}, so by Propositions \ref{singlenimskeletoniffdescendant} and \ref{nimskeletoniffdescendant} there is a unique molecule with descendant $T$. Thus $\mathcal{N}_n\subset D_n(T^*)$.
\end{proof} 

The following proposition should be clear, and we omit its proof.

\begin{proposition} If $T_1^*\neq T_2^*$, then $D_n(T_1^*)\cap D_n(T_2^*)=\varnothing$. In other words, any descendant class on $n$ vertices is disjoint from any other descendant class on $n$ vertices.
\end{proposition} 

% \begin{proof}
%     Suppose $D_n(T_1^*)\cap D_n(T_2^*)$ is nonempty. Then by Proposition \ref{nimskeletoniffdescendant} we have that $\text{Sk}_n(T_1^*)\cap \text{Sk}_n(T_2^*)$ is nonempty, meaning there is a tree $T$ such that applying Algorithm \ref{nimskeletonalg} on $T$ gives both $T_1^*$ and $T_2^*$, which contradicts uniqueness of a NIM skeleton for any given NIM tree $T$ (Proposition \ref{nimhasuniquenimskeleton}).
% \end{proof}

% Thus enumerating NIM trees can be done by enumerating all descendants of all possible molecules.

It then follows that the descendant classes partition the NIM trees. Thus, to count the NIM trees, it suffices to count the descendants of all possible molecules. 

\subsubsection{OGF for \texorpdfstring{$\mathcal{A}$}{A}, the NIM atoms}
Marking the number of vertices by $\mathcal{Z}$, the number of variable HDVs by $\mathcal{U}$, and the number of components by $\mathcal{J}$, we have the following combinatorial specification of $\mathcal{A}$:
\begin{equation*}
    \textstyle\mathcal{A}
    \cong 
    (\mathcal{Z}^4\times\mathcal{J})
    +
    (\mathcal{Z}^5\times \mathcal{U}\times \mathcal{J})+\bigcup_{k\geq 2}\left(\mathcal{Z}^{4k}\times \mathcal{U}^{k}\times \mathcal{J}\right).
\end{equation*} 
Let $A(z,u,r)$ be the OGF for $\mathcal{A}$. By the symbolic methods for OGFs \cite{flajolet_sedgewick_2013}, this specification translates  as $A(z,u,r)=z^4r+z^5ur+\sum_{k\geq 2}z^{4k}u^kr
% $, which simplifies to $\textstyle A(z,u,r)
=r z^4 \left(1 + u z + \frac{u^2 z^4}{1 - u z^4}\right)$. 
\subsubsection{OGF for \texorpdfstring{$\mathcal{N}^*$}{mathcal N star}, the molecules (NIM skeletons)}

We now find the OGF for the combinatorial class of molecules, $\mathcal{N}^*$ (equivalently, NIM skeletons by Proposition \ref{singlenimskeletoniffdescendant}). Of course, the atoms themselves are molecules. Moreover, any non-atomic molecule is a sequence of vertex-concatenated members of $\mathcal{A}$ sequence, under symmetric equivalence about the central path. Vertex-concatenation of two atoms from $\mathcal{A}$ results in the resulting vertex count being one less than the sum of that of each atom. Therefore, removal of one $\mathcal{Z}$ is required for each concatenation of an atom to a molecule under construction, which translates to the removal of $\mathcal{Z}$ for every atom in $\mathcal{A}$ that is used to construct a molecule of size $\geq 2$. 

We now take advantage of the fact that the number of bridge points between vertex-concatenated atoms in a molecule is exactly one more than the number of its component atoms. Since these bridge points are where we can add paths to the molecule to build descendants, \emph{we will now let $\mathcal{J}$ mark the bridge points between the atoms of the molecule}. Thus, in the remainder of the enumeration, we will adjoin $\mathcal{J}$ to each term when necessary to account for the single junction left out by counting only the total number of atomic components in the molecule.

For brevity, we adopt multi-index notation to write $\*x=(z,u,r)$ and $\*x^k=(z^k,u^k,r^k)$ for each $k\geq 1$. By the above comments, we have the combinatorial specification that $\mathcal{N}^*
    \cong
    \mathcal{J}\times
    \left(\mathcal{Z}\times\USeq_{\geq 1}(\mathcal{A}/\mathcal{Z})
    \right)$, 
which, by the symbolic method and Definition \ref{USEQ}, translates to the following OGF for $\mathcal{N}^*$, denoted $N^*(\*x)$:
\begin{equation*} \label{nstar}
     \textstyle N^*(\*x)= zr\left(\frac{1}{2}\frac{1}{1-A(\*x)/z}+\frac{1}{2} \frac{1+A(\*x)/z}{1-A\left(\*x^2\right)/z^2}-1\right).
\end{equation*}

\subsubsection{OGF for \texorpdfstring{$\mathcal{N}^*_s$}, the \emph{symmetric} molecules}
Denote $\mathcal{N}_{s}$ (resp. $\mathcal{N}_{a}$) as the set of NIM trees that are symmetric (resp. asymmetric) about the center of its central path. Let $\mathcal{N}^*_{s}$ (resp. $\mathcal{N}^*_{a}$) denote the combinatorial class of molecules that are symmetric (resp. asymmetric) about the center of its central path. We now provide a recursive specification for the symmetric molecules $\mathcal{N}^*_{s}$. 

Since $\mathcal{N}^*$ only counts asymmetric molecules under reflection about the central path as equivalent, we need to add an extra copy of the asymmetric molecules when their orientations with respect to the center are reversed. We take this into account below with the third term. 

Before we proceed, we remark on notation. Let $\mathcal{C}$ be a combinatorial class marked by three variables $\mathcal{Z}$, $\mathcal{U}$, and $\mathcal{J}$. For brevity, let $\left(\mathcal{C}\right)^2$ denote $\mathcal{C}\left[\mathcal{Z}\mapsto\mathcal{Z}^2,\mathcal{U}\mapsto\mathcal{U}^2,\mathcal{J}\mapsto\mathcal{J}^2\right]$. This is the combinatorial class formed from $\mathcal{C}$ by taking pairs of elements from $\mathcal{C}$, each pair containing the same element twice.\footnote{Note that this is different from $\mathcal{C}\times\mathcal{C}=\mathcal{C}^2$, which simply is the class of ordered pairs of \emph{any} two elements from $\mathcal{C}$, not ordered pairs of the same element} We then have the following specification:
\begin{equation*}
        \mathcal{N}^*_{s}
        \cong 
        \underbrace{\mathcal{J}\times\mathcal{A}}_{\text{atoms}}
        +     \underbrace{\left((\mathcal{N}^*)^2+(\mathcal{N}^*_{a})^2\right)/(\mathcal{Z}\times\mathcal{J})}_{\substack{\text{symmetric molecules composed of} \\ \text{two copies of a smaller molecule}}}
        +
        \underbrace{\mathcal{A}\times ((\mathcal{N}^*)^2+(\mathcal{N}^*_{a})^2)/(\mathcal{Z}^2\times\mathcal{J})}_{\substack{\text{symmetric molecules composed of} \\ \text{two copies of a smaller molecule} \\ \text{at opposite ends of a central atom}}}.
\end{equation*} 
Applying the simple observation that $\mathcal{N}^*_{a}=\mathcal{N}^*-\mathcal{N}^*_{s}$, the above simplifies as
\begin{equation*}
        \mathcal{N}^*_{s}
        \cong 
        \mathcal{J}\times\mathcal{A}
        +
        \left(2(\mathcal{N}^*)^2-(\mathcal{N}^*_{s})^2\right)/(\mathcal{Z}\times\mathcal{J})
        +
        \mathcal{A}\times \left(2(\mathcal{N}^*)^2-(\mathcal{N}^*_{s})^2\right)/(\mathcal{Z}^2\times\mathcal{J}).
\end{equation*} 
By symbolic methods this translates to the multivariate functional equation given by
\begin{equation}\label{ns}
    \textstyle N^*_{s}(\mathbf{x})=r A(\mathbf{x})+\frac{2N^*(\mathbf{x}^2)}{zr}\left(1+A(\mathbf{x})/z\right)
    -
    \frac{N^*_{s}(\mathbf{x}^2)}{zr}\left(1+A(\mathbf{x})/z\right),
\end{equation} 
from which we can extract coefficients. 

\subsubsection{OGF for \texorpdfstring{$\mathcal{N}_{s}$}{Ns}, the symmetric NIM \emph{trees}}

Any symmetric NIM tree $T$ can be constructed from two copies of some smaller NIM tree. There are four cases we must consider which correspond to the four possible combinations of parity in the variable HDVs $\mathcal{U}$ and the bridges $\mathcal{J}$. One of these four cases will next be shown to be impossible:

\begin{proposition}\label{nooddodd} No NIM tree $T$ has an odd number of variable HDVs and an odd number of bridges.
\begin{proof} Suppose for a contradiction any NIM skeleton is a member of $\mathcal{N}_{\text{sym,o,o}}$. There is an odd number of variable HDVs, so since the skeleton must be symmetric about the central path there must be a central variable HDV, a variable HDV that sits in the exact center of the central path. On the other hand, there is an odd number of junction points between atoms (including endpoints), so it follows that there must be a junction point between two atoms in the exact center of the central path. But a variable HDV has $\delta(v)\geq 3$ and a junction point/endpoint has $\delta(v)\leq 2$, a contradiction. \end{proof}
\end{proposition}  

Let $\mathcal{N}_{s,e,o}$ be the combinatorial class of symmetric NIM trees with an even number of variable HDVs and an odd number of bridges, and similarly define $\mathcal{N}_{s,e,e}$, $\mathcal{N}_{s,o,e}$, and $\mathcal{N}_{s,o,o}$ in the obvious way. $\mathcal{N}_{s,o,o}=\varnothing$ by Proposition \ref{nooddodd}, so the remaining three classes partition the symmetric NIM trees $\mathcal{N}_{s}$.
% It follows that the class of symmetric NIM trees $\mathcal{N}_s$ can be specified by
Hence $\mathcal{N}_s\cong\mathcal{N}_{s,e,e}+\mathcal{N}_{s,e,o}+\mathcal{N}_{s,o,e}$, and so
\begin{equation}\label{nsymsum}
    N_s(\*x)=N_{s,e,e}(\*x)+N_{s,e,o}(\*x)+N_{s,o,e}(\*x).
\end{equation}

\begin{proposition}\label{evenskelevennim} 
If $T\in\mathcal{N}_{s}$, then $T^*\in\mathcal{N}^*_{s}$.
\begin{proof} 
Suppose instead $T$ is symmetric but $T^*$ is asymmetric. $T$ is a result of adding pendants to $T$ at vertices marked $\mathcal{U}$ and adding paths at junctions marked $\mathcal{J}$, so by Proposition \ref{nooddodd} the number of $\mathcal{U}$s is even or the number of $\mathcal{J}$s is even. 

Suppose there is an even number of $\mathcal{U}$s. Then whatever pendants are added to one half of the central path must also be added to the corresponding HDV mirrored across the center of the central path. But there must be more HDVs on one side of the path than on the other, meaning this task is impossible. The same problem occurs in the case of an even number of $\mathcal{J}$s. Therefore, we can build every symmetric NIM tree from $\mathcal{N}^*_{s}$ as our foundation.
\end{proof}
\end{proposition} 

\begin{proposition}\label{oeinvariant} For any molecule $T^*$, any $T\in D(T^*)$ has the same number of variable HDVs and the same number of bridge/bridge points as $T^*$.
\begin{proof}
Fix a molecule $T^*$. We have by Proposition \ref{nimskeletoniffdescendant} that $T\in \text{Sk}(T^*)$, so applying algorithm \ref{nimskeletonalg} to $T$ gives $T^*$. Step 1 of the algorithm edge contracts bridges to a bridge point (which is a bridge), which leaves the total number of bridges unchanged. Step 2 swaps each path/singleton component of the graph induced by the variable HDVs in the molecule with $k$ HDVs to the unique atom in $\mathcal{A}$ with $k$ HDVs, thus leaving the total number of variable HDVs unchanged. \end{proof}
\end{proposition}

Taking terms with even powers of $u$ and even powers of $r$ from $N_s^*(z,u,r)$ by elementary techniques for generating functions, we obtain that
\begin{equation*}
    N^*_{s,e,e}(z,u,r)=
    \frac{1}{2}\left(\frac{1}{2}(N^*_{s}(z,u,r)+N^*_{s}(z,-u,r))+\frac{1}{2}(N^*_{s}(z,u,-r)+N^*_{s}(z,-u,-r))\right).
\end{equation*} 
We can further simplify things by noticing that the result of Proposition \ref{nooddodd} implies that a NIM tree with an even odd of variable HDVs can immediately be deduced to have an even number of bridge points and vice versa. Therefore, $\mathcal{N}_{s,e,o}\cong\mathcal{N}_{s,all,o}$ and $\mathcal{N}_{s,o,e}\cong\mathcal{N}_{s,o,all}$. We then have by elementary manipulation of generating functions that
\begin{align*}
    &N^*_{s,e,o}(z,u,r)=N^*_{s,all,o}(z,u,r)=\textstyle\frac{1}{2}(N^*_{s}(z,u,r)-N^*_{s}(z,u,-r)),
    \\
    &N^*_{s,o,e}(z,u,r)=N^*_{s,o,all}(z,u,r)=\textstyle\frac{1}{2}(N^*_{s}(z,u,r)-N^*_{s}(z,-u,r)).
\end{align*} Propositions \ref{evenskelevennim} and \ref{oeinvariant} imply that enumerating symmetric NIM trees with given parities of variable HDVs and bridges may be done by enumerating symmetric trees of the same parities in the descendant class $D(T^*)$. In other words, \begin{align*}
    \mathcal{N}_{s,e,e}&\cong\bigcup_{T^*\in \mathcal{N}^*_{s,e,e}}\{T:T\in D(T^*)\text{ and $T$ is symmetric}\},%\label{s*ee}
    % \\
    % \mathcal{N}_{s,e,o}&\cong\bigcup_{T^*\in \mathcal{N}^*_{s,e,o}}\{T:T\in D(T^*)\text{ and $T$ is symmetric}\},%\label{s*eo}
    % \\
    % \mathcal{N}_{s,o,e}&\cong\bigcup_{T^*\in \mathcal{N}^*_{s,o,e}}\{T:T\in D(T^*)\text{ and $T$ is symmetric}\}%\label{s*oe}.
    % ,
\end{align*} 
with analogous expressions for $\mathcal{N}_{s,e,o}$ and $\mathcal{N}_{s,o,e}$.
We now use the above specifications to find the OGFs for these respective three classes, recalling that their sum will give the OGF for $\mathcal{N}_s$ by Equation \eqref{nsymsum}. 

\begin{proposition} Let the OGF for $\mathcal{N}_{s,e,e}$, $\mathcal{N}_{s,e,o}$, and $\mathcal{N}_{s,o,e}$ be given by $N_{s,e,e}(z)$, $N_{s,e,o}(z)$, and $N_{s,o,e}(z)$, respectively. Then we have the following:
\begin{align}
    N_{s,e,e}(z)&= \textstyle N^*_{s,e,e}(z, \frac{1}{\sqrt{1 - z^2}},\frac{1}{\sqrt{1 - z^2}}),\label{nimsymee}
   \\
    N_{s,e,o}(z)&= \textstyle \frac{\sqrt{1 - z^2}}{1-z} N^*_{s,e,o}(z,\frac{1}{\sqrt{1 - z^2}}, \frac{1}{\sqrt{1 - z^2}}),\label{nimsymeo}
   \\
    N_{s,o,e}(z)&= \textstyle \frac{\sqrt{1 - z^2}}{1-z}N^{*}_{s,o,e}(z, \frac{1}{\sqrt{1 - z^2}},\frac{1}{\sqrt{1 - z^2}}).\label{nimsymoe}
\end{align}
\begin{proof} 
We first prove \eqref{nimsymee}. Fix a molecule $T^*\in\mathcal{N}_{s,e,e}$. Fix some symmetric $T\in D(T^*)$. Then $T^*$ must be symmetric by Proposition \ref{evenskelevennim}. Moreover, $T$ must have the same number of variable HDVs ($\mathcal{U}$) and the same number of bridges ($\mathcal{J}$) as $T^*$ by Proposition \ref{oeinvariant}. Thus, if $T$ has $k$ variable HDVs, $\mathcal{U}^k$, we can entirely determine $T$ using only one half of its central path since the other side must be its reflection about the center of the central path. $D(T^*)$ can therefore be counted by vectors with nonnegative entries in $\mathbb{Z}^{k/2}$, each entry indicating how many pendants must be added to the molecule $T^*$ at that variable HDV to achieve the number found on that of $T$. These are given by the specification $\Seq_{k/2}\left(\mathcal{Z}^2\right)$, the square coming from the fact that one $\mathcal{Z}$ on a variable HDV immediately implies that $\mathcal{Z}$ be on the variable HDV reflected across the central path as well. From this, we obtain the following specification: 
    \begin{equation*}
        \mathcal{N}_{s,e,e}\cong\mathcal{N}^*_{s,e,e}\left[\mathcal{\mathcal{U}}^2\mapsto\Seq\left(\mathcal{Z}^2\right),\mathcal{\mathcal{J}}^2\mapsto\Seq\left(\mathcal{Z}^2\right)\right],
    \end{equation*} 
    which translates to the OGF $\textstyle N_{s,e,e}(z)=N^*_{s,e,e}(z, \frac{1}{\sqrt{1 - z^2}},\frac{1}{\sqrt{1 - z^2}})$ as claimed.
    
    We now prove \eqref{nimsymeo}. We can apply the same reasoning from when specifying $\mathcal{N}_{s,e,e}$ from $\mathcal{N}^*_{s,e,e}$, except that what were two halves of the central path in the $\mathcal{N}_{s,e,e}$ case are now two symmetric components attached by a central bridge. That is, that now there is an odd number of bridges, there must be a central bridge. This central bridge has no counterpart reflected about the center of the central path---it \emph{is} the center of the central path---so a bridge of any length can be added freely without affecting anything else. We therefore have the specification
    \begin{equation*} \mathcal{N}_{s,e,o}\cong\Seq(\mathcal{Z})\times\mathcal{N}^*_{s,e,o}/\mathcal{J}\left[\mathcal{\mathcal{U}}^2\mapsto\Seq\left(\mathcal{Z}^2\right),\mathcal{\mathcal{J}}^2\mapsto\Seq\left(\mathcal{Z}^2\right)\right],
    \end{equation*} where removal of a $\mathcal{J}$ corrects for the central bridge already accounted for by the leading $\Seq(\mathcal{Z})$. This translates to the OGF $\textstyle N_{s,e,o}(z)=
    \frac{\sqrt{1 - z^2}}{1-z}
    N^{*}_{s,e,o}(z, \frac{1}{\sqrt{1 - z^2}},\frac{1}{\sqrt{1 - z^2}})$ as claimed.
    
    We conclude by proving \eqref{nimsymoe}. We can apply the same reasoning from when specifying $\mathcal{N}_{s,e,e}$ from $\mathcal{N}^*_{s,e,e}$, except that now there is an odd number of variable HDVs, and thus there must be a central variable HDV. This central variable HDV has no counterpart reflected about the center of the central path---it \emph{is} the center of the central path---so its pendants can be added freely without affecting anything else. We therefore have the specification $$\mathcal{N}_{s,o,e}\cong\Seq(\mathcal{Z})\times\mathcal{N}^*_{s,o,e}/\mathcal{U}\left[\mathcal{\mathcal{U}}^2\mapsto\Seq\left(\mathcal{Z}^2\right),\mathcal{\mathcal{J}}^2\mapsto\Seq\left(\mathcal{Z}^2\right)\right],$$
      where removal of a $\mathcal{U}$ corrects for the extra central HDV already accounted for by the leading $\Seq(\mathcal{Z})$. This translates to the OGF $\textstyle N_{s,o,e}(z)=
        \frac{\sqrt{1 - z^2}}{1-z}
         N^{*}_{s,o,e}(z, \frac{1}{\sqrt{1 - z^2}},\frac{1}{\sqrt{1 - z^2}})$ as claimed.
\end{proof}
\end{proposition}

We obtain that the OGF for the symmetric NIM trees, $N_s(z)$, is given by 
\begin{multline}\label{nsym}
  \textstyle N_{s}(z)
   =
   N^*_{s,e,e}(z, \frac{1}{\sqrt{1 - z^2}},\frac{1}{\sqrt{1 - z^2}})
   + 
   \frac{\sqrt{1 - z^2}}{1-z}
   N^*_{s,e,o}(z,\frac{1}{\sqrt{1 - z^2}}, \frac{1}{\sqrt{1 - z^2}})
   +\textstyle
   \frac{\sqrt{1 - z^2}}{1-z}
     N^{*}_{s,o,e}(z, \frac{1}{\sqrt{1 - z^2}},\frac{1}{\sqrt{1 - z^2}}).
\end{multline}

\subsubsection{OGF for \emph{asymmetric} NIM trees, \texorpdfstring{$\mathcal{N}_a$}{Na}}

Since any descendant of an asymmetric tree is asymmetric itself, we have that all asymmetric NIM trees that are descendants of \emph{asymmetric} molecules are given by the descendants of anything in $\mathcal{N}_a^*$, which is specified as simply $\textstyle\mathcal{N}^*_{a}\left[\mathcal{U}\mapsto \Seq(\mathcal{Z}),\mathcal{J}\mapsto\Seq(\mathcal{Z})\right]$.
% \begin{equation}\label{nimasymmoleculenimasym}
%     \textstyle\mathcal{N}^*_{a}\left[\mathcal{U}\mapsto \Seq(\mathcal{Z}),\mathcal{J}\mapsto\Seq(\mathcal{Z})\right].
% \end{equation}
The remaining asymmetric NIM trees must then be descendants of symmetric molecules. Work is made much simpler, however, by observing that asymmetric descendants of symmetric molecules are exactly the result of taking half of what remains after starting with the total number of ways to make NIM trees \emph{without} regards to symmetry and subtracting out the ways that result in symmetric trees, the latter of which we have already classified. In other words, we have the following specification for $\mathcal{N}_{a}$: 
\begin{align*}
   \textstyle 
   \mathcal{N}_{a}
   \cong
   \mathcal{N}^*_{a}\left[\mathcal{U}\mapsto \Seq(\mathcal{Z}),\mathcal{J}\mapsto\Seq(\mathcal{Z})\right]\nonumber
   &+
   \textstyle \frac{1}{2}\left(\mathcal{N}^*_{s,e,e}\left[\mathcal{U}\mapsto \Seq(\mathcal{Z}),\mathcal{J}\mapsto\Seq(\mathcal{Z})\right] -\mathcal{N}_{s,e,e}\right)\nonumber
   \\
   &+
   \textstyle\frac{1}{2}\left(\mathcal{N}^*_{s,e,o}\left[\mathcal{U}\mapsto \Seq(\mathcal{Z}),\mathcal{J}\mapsto\Seq(\mathcal{Z})\right]-\mathcal{N}_{s,e,o}\right)
   \\
   &+
   \textstyle\frac{1}{2}\left(\mathcal{N}^*_{s,o,e}\left[\mathcal{U}\mapsto \Seq(\mathcal{Z}),\mathcal{J}\mapsto\Seq(\mathcal{Z})\right]-\mathcal{N}_{s,o,e}\right)\nonumber.
\end{align*} This specification together with Equations \eqref{nimsymee}, \eqref{nimsymeo}, and \eqref{nimsymoe} above then translate to the below OGF for $\mathcal{N}_{a}$, which we denote by $N_a(z)$.
\begin{align}\label{nasym}
    \textstyle N_{a}(z)
    =\textstyle 
    N^*_{a}(1,\frac{1}{1-z},\frac{1}{1-z}) \nonumber
    &\textstyle +
    \frac{1}{2}
    \left(N^*_{s,e,e}(z,\frac{1}{1-z},\frac{1}{1-z})-N^*_{s,e,e}(z,\frac{1}{\sqrt{1-z^2}},\frac{1}{\sqrt{1-z^2}})\right) \nonumber
    \\
    &\textstyle +
    \frac{1}{2} \left(N^*_{s,e,o}(z,\frac{1}{1-z},\frac{1}{1-z})-\frac{\sqrt{1-z^2}}{1-z} N^*_{s,e,o}(z,\frac{1}{\sqrt{1-z^2}},\frac{1}{\sqrt{1-z^2}})\right) 
    \\
    &\textstyle +
    \frac{1}{2} \left(N^*_{s,o,e}(z,\frac{1}{1-z},\frac{1}{1-z})-\frac{\sqrt{1-z^2}}{1-z} N^*_{s,o,e}(z,\frac{1}{\sqrt{1-z^2}},\frac{1}{\sqrt{1-z^2}})\right). \nonumber
\end{align} 

\subsubsection{Final Result}
The symmetric NIM trees $\mathcal{N}_s$ and the asymmetric NIM trees $\mathcal{N}_a$ together partition all NIM trees $\mathcal{N}$, so $\mathcal{N}\cong\mathcal{N}_s+\mathcal{N}_a$. We then have that the OGF for $\mathcal{N}$, which we denote by $N(z)$, is given by the sum of \eqref{nsym} and \eqref{nasym}. Finally adding in the paths (including the singleton (which has OGF $\frac{z}{1-z}$) to the sum of $N_{s}(z)$ and $N_{a}(z)$, $N(z)$ simplifies as:
\begin{align}
 \textstyle N(z)=\frac{z}{1-z}+&\textstyle\frac{1}{2}N^*_{s,e,e}(z,\frac{1}{\sqrt{1-z^2}},\frac{1}{\sqrt{1-z^2}})+\frac{1}{2}N^*_{s,e,e}(z,\frac{1}{1-z},\frac{1}{1-z})\nonumber
 \\
 &+
 \textstyle\frac{\sqrt{1-z^2}}{2(1-z)}\textstyle N^*_{s,e,o}(z,\frac{1}{\sqrt{1-z^2}},\frac{1}{\sqrt{1-z^2}})+\frac{1}{2}N^*_{s,e,o}(z,\frac{1}{1-z},\frac{1}{1-z}) \label{nimgenfunctionugly}
 \\
 &+
 \textstyle\frac{\sqrt{1-z^2}}{2(1-z)}\textstyle N^*_{s,o,e}(z,\frac{1}{\sqrt{1-z^2}},\frac{1}{\sqrt{1-z^2}})+\frac{1}{2}N^*_{s,o,e}(z,\frac{1}{1-z},\frac{1}{1-z}),\nonumber
\end{align} 
from which coefficients can be extracted.

\subsubsection{Asymptotics}
We can employ singularity analysis (\cite{flajolet_sedgewick_2013}, Chapter IV) to study the asymptotic behavior of $N(n)$ and obtain asymptotic ratio of NIM trees to caterpillars. Let the combinatorial class $\mathcal{N}_{\text{ord}}$ be the class of ordered NIM trees, which are simply NIM trees but are distinguished by their orientation from left to right so that each asymmetric regular NIM tree corresponds to two separate ordered NIM trees. The class of symmetric NIM trees $\mathcal{N}_{\text{sym}}$ is asymptotically vanishing in $\mathcal{N}$, so $\lim_{n\to\infty}\frac{N(n)}{N_{\text{ord}}(n)}=\frac{1}{2}$.

\begin{theorem}\label{onehalf} 
$N(n)\underset{n\to\infty}{\sim}c^n$, where $c=1.8265...$.
\begin{proof} 
$\mathcal{N}_{\text{ord}}$ is specified as 
\begin{equation*}
    \mathcal{N}_{\text{ord}}\cong\Seq_{\geq 1}(\mathcal{Z})+\Seq_{\geq 1}\left(\left(\mathcal{A}\circ\left(\mathcal{Z},\Seq(\mathcal{Z}),\Seq(\mathcal{Z})\right)\right)/\mathcal{Z}\right),
\end{equation*}
which translates to the following OGF:
\begin{align*}
    N_{\text{ord}}(z)
    &=\frac{z}{1-z}+\frac{A(z,\frac{1}{1-z},\frac{1}{1-z})/z}{1-A(z,\frac{1}{1-z},\frac{1}{1-z})/z}
    =\frac{z \left(z^2+z-1\right) \left(z^3-2 z^2+z-1\right)}{1-3z+3z^2-2z^3+2z^5-z^6},
\end{align*}
where $\frac{z}{1-z}$ accounts for the singleton and the simple paths. Applying the methods of singularity analysis (see Chapter IV in \cite{flajolet_sedgewick_2013}), it follows since the smallest pole in absolute value is the solution $\rho$ to the equation $1-3z+3z^2-2z^3+2z^5-z^6= 0$ with $\rho=0.54749048...$, that $N_{\text{ord}}(n)\sim (1/\rho)^n=(1.8265...)^n$, so since $\lim_{n\to\infty}\frac{N(n)}{N_{\text{ord}}(n)}=\frac{1}{2}$ the result follows. 
\end{proof}
\end{theorem}

\begin{corollary} Let $C(n)$ denote the number of caterpillar graphs on $n$ vertices. Then $\lim_{n\to\infty}\frac{N_{\text{ord}}(n)}{C(n)}=0$.
\begin{proof} 
$C(n)=2^{n-4}+2^{\lfloor (n-4)/2\rfloor}$ is known from \cite{FrankSchwenk}, so the result is immediate by Theorem \ref{onehalf}.
\end{proof}
\end{corollary}

Computation analysis yields the following conjecture, which has been checked up to $n=300$:

\begin{conjecture}\label{nimrecurrence} $N(n)$, the number of NIM trees on $n$ vertices, satisfies the following 15-depth linear recurrence for $n\geq 16$.
\begin{align}
    N(n)=& \  2N(n - 1) + N(n - 2) - 3 N(n - 3)
    +2 N(n - 4) -N(n - 5)\nonumber
    \\
    &- 2 N(n - 6)+ N(n - 7) +3 N(n - 8) - 4 N(n - 9) - N(n - 10)
    \\
    &+ 2 N(n - 11) - 2 N(n - 12)+2 N(n - 13) + N(n - 14) - N(n - 15).\nonumber
\end{align}
\end{conjecture} 

\bibliographystyle{plain}
\bibliography{bibliography.bib}

\end{document}